\documentclass[11pt,fleqn,leqno]{siamltex}

\usepackage{amsmath}
\usepackage{amssymb}
\usepackage{graphicx}
\usepackage{url}

\newtheorem{remark}[theorem]{Remark}
\newtheorem{example}[theorem]{Example}
\newtheorem{summary}[theorem]{Summary}

\newcommand{\C}{\mathbb C}

\newcommand{\bx}{x}

\newcommand{\bz}{z}
\newcommand{\zero}{0}
\newcommand{\smath}[1]{{\mbox{\scriptsize $#1$}}}

\catcode`@=11     % Make @ act like a letter
\font\tenex=cmex10 % math extension
\newdimen\p@renwd
\setbox0=\hbox{\tenex B} \p@renwd=\wd0 % width of the big left [
\def\bmat#1{\begingroup \m@th
  \setbox\z@\vbox{\def\cr{\crcr\noalign{\kern2\p@\global\let\cr\endline}}%
    \ialign{$##$\hfil\kern2\p@\kern\p@renwd&\thinspace\hfil$##$\hfil
      &&\quad\hfil$##$\hfil\crcr
      \omit\strut\hfil\crcr\noalign{\kern-\baselineskip}%
      #1\crcr\omit\strut\cr}}%
  \setbox\tw@\vbox{\unvcopy\z@\global\setbox\@ne\lastbox}%
  \setbox\tw@\hbox{\unhbox\@ne\unskip\global\setbox\@ne\lastbox}%
  \setbox\tw@\hbox{$\kern\wd\@ne\kern-\p@renwd\left[\kern-\wd\@ne
    \global\setbox\@ne\vbox{\box\@ne\kern2\p@}%
    \vcenter{\kern-\ht\@ne\unvbox\z@\kern-\baselineskip}\,\right]$}%
  \null\;\vbox{\kern\ht\@ne\box\tw@}\endgroup}
\catcode`@=12    % @ is no longer a letter

\author{Michiel~E.~Hochstenbach\thanks{%
Version \today.
Department of Mathematics and Computer Science,
TU Eindhoven,
PO Box 513, 5600 MB, The Netherlands,
{\tt www.win.tue.nl/$\sim$hochsten}.
This author has been supported by an NWO Vidi research grant.}
\and
Christian Mehl\thanks{%
Institut f\"{u}r Mathematik, Technische Universit\"at Berlin, Sekretariat MA 4-5,
Stra\ss e des 17.~Juni 136, 10623 Berlin, Germany,
{\tt mehl@math.tu-berlin.de}.
This author has been supported by a Dutch 4TU AMI visitor's grant.}
\and
Bor~Plestenjak\thanks{%
IMFM and Faculty of Mathematics and Physics, University of Ljubljana,
Jadranska 19, SI-1000 Ljubljana, Slovenia,
{\tt bor.plestenjak@fmf.uni-lj.si}.
This author
has been supported in part by the Slovenian Research Agency (grant P1-0294).}}

\title{Solving singular generalized eigenvalue problems \\
by a rank-completing perturbation}

\begin{document}
\maketitle

\begin{abstract}
Generalized eigenvalue problems involving a singular pencil are very challenging
to solve, both with respect to accuracy and efficiency.
The existing package Guptri is very elegant but may be time-demanding, even for
small and medium-sized matrices.
We propose a simple method to compute the eigenvalues of singular pencils,
based on one perturbation of the original problem of a certain specific rank.
For many problems, the method is both fast and robust.
This approach may be seen as a welcome alternative to staircase methods.
%For some cases, it can give good results in situations in which staircase type methods fail.
%{\bf TODO: please check this sentence, it is not too strong?}
\end{abstract}



\pagestyle{myheadings}
\thispagestyle{plain}
\markboth{HOCHSTENBACH, MEHL, AND PLESTENJAK}{SOLVING SINGULAR GEPs BY A RANK-COMPLETING PERTURBATION}

\section{Introduction}
\label{sec1}
We study the computation of eigenvalues of small to medium-sized matrix pencils
$A-\lambda B$, where $A$ and $B$ are (real or complex) $n\times m$ matrices such that
the matrix pencil $A-\lambda B$ is singular, which means that $m\neq n$, or
if $m=n$ then
\[
\det(A-\lambda B) \equiv 0.
\]
In these cases, the common definition of eigenvalues as roots of $\det(A-\lambda B)$ would
only be meaningful for the case $m=n$, but turns out to be useless as any value $\lambda\in\mathbb C$
would be an eigenvalue.
Therefore, \emph{finite eigenvalues} of a singular matrix pencil
$A-\lambda B$ are typically defined as values $\lambda_0\in\mathbb C$ satisfying
$\text{rank}(A-\lambda_0 B) < \text{nrank}(A,B)$, where
\[
\text{nrank}(A,B) := \max_{\zeta \in \C} \, \text{rank}(A-\zeta B)
\]
denotes the \emph{normal rank} of the pencil $A-\lambda B$; see \cite{Dem00}. Similarly, we say that
$\infty$ is an eigenvalue of the singular pencil $A-\lambda B$ if
$\text{rank}(B)<\text{nrank}(A,B)$. In the following we will mainly restrict ourselves to the case $m=n$ as
the case $m\neq n$ can easily be reduced to the square case by adding an appropriate number of
zero rows or columns.
%(Proposition~\ref{prop:genpert} will deal with the rectangular case.)

The singular generalized eigenvalue problem (singular GEP) is well known to be ill-conditioned as
arbitrarily small perturbation may cause drastic changes in the eigenvalues. A classical
example is given by the pencils
\[
A-\lambda B=\left[\begin{array}{cc}1&0\\ 0&0\end{array}\right]
-\lambda\left[\begin{array}{cc}1&0\\ 0&0\end{array}\right]\ \ \mbox{and}\ \
\widetilde A-\lambda\widetilde B=
\left[\begin{array}{cc}1&\varepsilon_1\\ \varepsilon_2&0\end{array}\right]
-\lambda\left[\begin{array}{cc}1&\varepsilon_3\\ \varepsilon_4&0\end{array}\right],
\]
where $\varepsilon_1, \dots,\varepsilon_4\in\mathbb C\setminus\{0\}$; see \cite{Kag00}.
While $A-\lambda B$ is singular and has only the eigenvalue $1$, the perturbed pencil $\widetilde A-\lambda\widetilde B$
is regular and has the eigenvalues $\frac{\varepsilon_1}{\varepsilon_3}$ and
$\frac{\varepsilon_2}{\varepsilon_4}$ that can be anywhere in the complex plane even for tiny absolute values
of $\varepsilon_1, \dots,\varepsilon_4$.

On the other hand, it was observed in \cite{Wil79} that a situation as above is exceptional
and that generically small perturbations of a singular square pencil make the pencil regular and
some of the eigenvalues of the perturbed pencil are very close to the original eigenvalues of the
singular pencil. The following example illustrates this. The Matlab commands
\begin{verbatim}
    A = diag([1 2 3 0 0 0]);
    B = diag([2 3 4 0 0 0]);
    eig(U'*A*V, U'*B*V)
\end{verbatim}
where $U$ and $V$ are certain random $6 \times 6$ orthogonal matrices, compute the following eigenvalues:
\begin{verbatim}
    0.5000  0.6667  0.7500  0.1595  0.6756  0.6543
\end{verbatim}
We see that the three (finite) eigenvalues of the regular part are correct.
Following the terminology of \cite{VD79}, the other three values are ``fake eigenvalues''
and correspond to the singular part of the pencil.
(Explicit error analysis for the eigenvalues of singular pencils
has been undertaken in \cite{DDM08,DKa87}.) Despite this observation, Van Dooren suggests in
\cite{VD79} to solve the singular generalized eigenvalue problem by first extracting the
regular part and then use the QZ algorithm on that part.
Wilkinson strongly supports this recommendation in \cite{Wil79}.

A robust software package which follows Van Dooren's approach
is Guptri~\cite{DKa93, Guptri}. For a singular pencil, first a ``staircase'' algorithm is applied
to deflate the singular part of the pencil, and then the QZ algorithm is used to compute the
eigenvalues of the remaining regular part. While the results of Guptri are usually excellent,
this method may be quite time-consuming; for instance, applying Guptri
on a singular $300 \times 300$ pencil on our machine took over 20 seconds, while Matlab's
{\sf eig} on a random pencil of the same size spent less than a
second.\footnote{We note that this experiment has been performed some years ago.
A current practical issue is that there is no publicly available 64-bit Guptri code.}
Another issue is the fact that staircase type methods such as Guptri need rank decisions.
%and sometimes ill-posed, in particular for large singular blocks.
If the pencil has a minimal index of size $\eta$ (see Section~\ref{sec:prelim} for more details),
then at least $\eta+1$ such decisions have to be taken.
Typically, these decisions tend to become more and more critical during a
run of the staircase algorithm. See, e.g., \cite{EM} or \cite[Ex.~18]{MPl14}, where a
variant of the staircase algorithm for the singular two-parameter eigenvalue problem,
introduced in \cite{MPl09}, fails in double precision but gets the right result in higher precision.

Another way of extracting the regular part using
fewer rank decisions has been suggested in \cite{Meh18}.
One may view the singular pencil as a constant coefficient differential-algebraic
equation and perform a regularization procedure with the help of a derivative array
as described in \cite{BenLMV15}. In this way, the regular part of the pencil can be extracted
by only three nullspace computations. However, the
derivative array approach leads to an inflation of the system by a factor of at least $\eta+1$,
where $\eta$ is the largest minimal index of the given pencil, and may thus result in
high computational costs.

We propose a new method to compute the eigenvalues
of a singular pencil. The method is based on considering perturbations of rank
\[
k = n - \text{nrank}(A,B)
\]
which we will call \emph{rank-completing perturbations} as the rank is exactly large
enough to generically turn the pencil into a pencil of full normal rank. As we
will show, the canonical form of the original regular part of the given
pencil stays invariant under generic rank-completing perturbations.

The idea of computing eigenvalues of singular pencils with rank-completing perturbations
is not completely new, and the following specific type has been used in system theory as early as in the 70s
(without the use of the terminology ``rank-completing perturbation''). If a linear time-invariant control system of the form
\begin{align*}
\dot x&=Ax+Bu,\\
y&=Cx+Du
\end{align*}
is given, where $A\in\mathbb R^{n,n}$, $B\in\mathbb R^{n,m}$, $C\in\mathbb R^{r,n}$, and $D\in\mathbb R^{r,m}$
are the system matrices, $x$ stands for the state of the system, $u$ is the input, and
$y$ is the output, then the eigenvalues of the system pencil
\[
S(\lambda)=\left[\begin{array}{cc}\lambda I-A&B\\ -C&D\end{array}\right]
\]
are of particular interest in control theory; see \cite{EmaVD82} and the references therein.
(If the system is minimal, then these eigenvalues are also referred to as \emph{transmission zeros}
of the system.) Clearly, if $m\neq r$, then the pencil $S(\lambda)$ is rectangular and thus singular.
For that case and under the additional assumptions $r<m$ and $\operatorname{nrank}S(\lambda)=n+r$,
the following algorithm based on ideas of \cite{DavW74} has been proposed in \cite{LauM78}
for the computation of the transmission zeros:
\medskip

\begin{tabular}{ll}
{\footnotesize 1:} & Select random matrices $[E_1 \ \, F_1]$, $[E_2 \ \, F_2]\in\mathbb
R^{m-r,n+m}$ so that\\
& $S_i(\lambda):=\left[\begin{array}{cc}\lambda I-A&B\\ -C&D\\ E_i&F_i\end{array}\right]$\\
& is regular for $i=1,2$.
\\
{\footnotesize 2:} & Compute the eigenvalues ${\mathcal E}_i$ of $S_i(\lambda)$ for $i=1,2$. \\
{\footnotesize 3:} & Compute the intersection
${\mathcal E} = {\mathcal E}_1 \cap {\mathcal E}_2$.
\end{tabular}
\medskip

Since for each eigenvalue $\lambda_0$ of $S(\lambda)$
we have %we obviously have
$\operatorname{rank}S_i(\lambda_0)<n+m$, it immediately follows that
the eigenvalues of $S(\lambda)$ are contained in the spectrum of $S_i(\lambda)$ for both $i=1,2$.
The extended matrix pencils will
give rise to two sets of fake eigenvalues. As generically these sets will be
disjoint if the applied perturbations are generated randomly, it follows that
the set $\mathcal E$ will generically coincide with the set of eigenvalues of
$S(\lambda)$.

However, as pointed out in \cite{EmaVD82}, this method may encounter difficulties
in distinguishing the finite zeros from the infinite ones, in particular if the latter
occur with a high multiplicity.
Another problem may occur in identifying the values that belong to the intersection $\mathcal E$.
Although the original eigenvalues of the pencil theoretically coincide with a
subset of both ${\mathcal E}_1$ and ${\mathcal E}_2$,
they may still differ slightly in practice due to finite precision arithmetic.
Therefore, a tolerance has to be prescribed that decides when two values are considered
to be equal. If this tolerance is chosen too small, then some of the eigenvalues
may be missed. If, on the other hand, the tolerance is set too large, then
two close fake eigenvalues of $S_1(\lambda)$ and $S_2(\lambda)$ may be falsely identified as an eigenvalue
of $S(\lambda)$.
%{\bf TODO: It would be nice to have an example in the numerical experiments that
%shows both effects simultaneously for a given tolerance showing that this method
%can drastically fail.}

In this paper, we show that the eigenvalues of a singular pencil can be
efficiently computed with the help of {\em just one} rank-completing perturbation of the form
\[
A-\lambda B+\tau \, U(D_A-\lambda D_B)V^*,
\]
where $U$, $V$ are $n \times k$ matrices with orthonormal columns,
$D_A,D_B$ are diagonal $k\times k$ matrices,
and $\tau$ is a nonzero scalar. The orthonormality of the columns is not strictly necessary,
but convenient, for instance since in this case the norm of the perturbation can
easily be controlled by the parameter $\tau$.
The problem of identifying the subset of eigenvalues of the original pencil among the computed
eigenvalues of the perturbed pencil is then taken care of by the \emph{key observation}
that the left and right eigenvectors that correspond to the true eigenvalues satisfy orthogonality
relations with respect to the matrices $U$ and $V$.
Thus, instead of comparing the spectra
of two different pencils, the true eigenvalues can be separated from the fake
eigenvalues by using information from the corresponding left and right eigenvectors
from only one perturbed pencil.
%The accuracy of the new approach may not be optimal in all cases;
%see the numerical experiments in Section~\ref{sec:num}.
%However, we stress that computing the spectrum of a singular pencil is
%not a simple task at all, and that for several applications, Guptri
%may be unacceptably slow.
We note that perturbations of singular matrix pencils have already been considered
in \cite{BHM98, DDM08, MMW15, Val12,Val13},
but it seems that a detailed investigation of rank-completing perturbations is new,
except for \cite{MehMW16}, where the case of singular Hermitian pencils of normal rank $n-1$ was
considered.

%{\bf TODO: some words on the stability of the method}
%{\bf TODO: what papers to cite? for instance those by Christian, and Spanish ones?
%Bor: Perhaps we should add the paper by Sifuentes, Gimbutas, Greengard
%on Randomized methods for rank-deficient systems. This was presented on HS in Spa.}

The rest of this paper is organized as follows.
After some preliminaries in Section~\ref{sec:prelim},
we review some motivating applications where one is interested in computing
eigenvalues of a singular matrix pencil in Section~\ref{sec:appl}.
The main theoretical results are presented in
Section~\ref{sec:rankcomplete}, while the numerical method based
on these results is introduced in Section~\ref{sec:method}, followed
by some numerical experiments in Section~\ref{sec:num}. In Section~\ref{sec:2EP} we discuss
singular two-parameter eigenvalue problems and present a new numerical method for
such problems. We summarize some conclusions in Section~\ref{sec:concl}.

\section{Preliminaries}
\label{sec:prelim}
We will interpret matrix pencils both as pairs of
matrices $(A,B)\in\mathbb C^{n,m}\times \mathbb C^{n,m}$ or as $n\times m$ matrix polynomials
$A-\lambda B$ of degree at most one and we will switch between these notations whenever useful.
An important tool in the theory of singular pencils is the
Kronecker canonical form (KCF) of a pencil $A-\lambda B$; see, e.g., \cite{Gan59}.

\begin{theorem}[Kronecker canonical form]\label{thm:knf}
Let $A-\lambda B$ be a complex $n\times m$ matrix pencil. Then there exist nonsingular matrices
$P\in\mathbb C^{n,n}$ and $Q\in\mathbb C^{m,m}$ such that
\begin{equation}\label{eq:knf}
P(A-\lambda B)Q=\left[\begin{array}{cc} R(\lambda)&0\\ 0 & S(\lambda)\end{array}\right], \qquad
R(\lambda)=\left[\begin{array}{cc}J-\lambda I_r&0\\ 0&I_s-\lambda N\end{array}\right]
\end{equation}
%where
%\[
%\mathcal R(\lambda)=\left[\begin{array}{cc}J-\lambda I_r&0\\ 0&I_s-\lambda N\end{array}\right]
%\]
with $J$, $N$ in Jordan normal form and in addition $N$ being nilpotent, and
\[
S(\lambda)=
\text{diag}\big(L_{m_1}(\lambda),\dots,L_{m_k}(\lambda),\, L_{n_{1}}(\lambda)^\top,\dots,L_{n_{\ell}}(\lambda)^\top\big),
\]
where
$L_{j}(\lambda)=[0 \ \, I_{j}]-\lambda \, [I_{j} \ \, 0]$
is of size $j\times (j+1)$, and $m_i\ge 0$ for $i=1,\dots,k$,
and $n_i\ge 0$ for $i=1,\dots,\ell$.
\end{theorem}

The pencils $R(\lambda)$ and $S(\lambda)$ in Theorem~\ref{thm:knf} are called the
\emph{regular} and the \emph{singular part} of $A-\lambda B$, respectively. The eigenvalues of $J$ are exactly the
finite eigenvalues of $A-\lambda B$, while the eigenvalue $0$ of $N$ corresponds to the infinite eigenvalue
of $A-\lambda B$. The parameters $m_1,\dots,m_k$ and
$n_{1},\dots,n_{\ell}$ are called the \emph{right} and the \emph{left minimal indices} of $A-\lambda B$, respectively.
One may easily check that the normal rank
$\text{nrank}(A,B)$
%of $A-\lambda B$
is equal to
%is then given by
$\min(n-\ell,m-k)$. In the remainder of the paper, we will consider square pencils $A-\lambda B$, i.e.,
we have $n=m$. Note that this implies $k=\ell$, i.e., we must have
the same number of right and left minimal indices. However, the particular
values of the left and right minimal indices may be distinct.

In contrast to the eigenvalues of singular pencils, the corresponding eigenvectors and deflating subspaces
are not well defined. To understand why, we consider the following example borrowed from \cite{MehMW18}
(and slightly adapted). The pencil
\begin{equation}\label{exdefl}
A-\lambda B=\left[\begin{array}{ccc}1-\lambda& \phantom-0&0\\ 0&-\lambda&1\\ 0& \phantom-0&0\end{array}\right]
\end{equation}
obviously has the regular part $R(\lambda)=[1-\lambda]$ and thus the pencil $A-\lambda B$ has
the single eigenvalue $\lambda_0=1$ with algebraic multiplicity one. Nevertheless, any vector of the
form $x(\alpha,\beta):=[\alpha \ \, \beta \ \, \beta]^\top$ with $\alpha,\beta\in\mathbb C$ satisfies
$Ax=\lambda_0 Bx$
and thus could be interpreted as an eigenvector of the pencil. One may argue that %in this example
the choice $\alpha\neq 0$ and $\beta=0$ seems to be canonical and gives %rise to
a unique one-dimensional deflating
subspace ``corresponding'' to the regular part of the pencil. But on the other hand it follows from
the equality
\[
\left[\begin{array}{ccc} \phantom-1/\alpha & 0 &0\\ -\beta/\alpha & 1&0\\ \phantom-0&0&1 \end{array}\right]\left[\begin{array}{crc}1-\lambda&0&0\\ 0&-\lambda&1\\ 0&0&0\end{array}\right]
\left[\begin{array}{ccc} \alpha & 0 &0 \\ \beta & 1 & 0\\ \beta &0 & 1\end{array}\right]=\left[\begin{array}{crc}1-\lambda&0&0\\ 0&-\lambda&1\\ 0&0&0\end{array}\right]
\]
that for any choice of $\alpha,\beta$ with $\alpha\neq 0$, the vector $x(\alpha,\beta)$ can be
used to extract the regular part of the pencil as well (and thus could also be considered as ``corresponding'' to the
regular part). For this reason, we restrict ourselves to the computation
of eigenvalues of singular pencils, but do not consider corresponding eigenvectors.
However, as we will see in Section~\ref{sec:rankcomplete}, the eigenvectors of the
{\em perturbed} pencil play a key role in our approach.
%(2) for the MEP applications, the eigenvalues are what matters, and we can compute the relevant
%eigenvectors of the MEP after we have computed the eigenvalues, I think we should mention this} -
%BP: No, even for MEP we usually do not have eigenvectors, for instance in biroots pencils A-lambda B-mu C are singular as well.
%In any case, for a singular MEP we cannot do anything with eigenvectors of Delta matrices.

Instead of eigenvectors and deflating subspaces, the concept of \emph{reducing subspaces} introduced in
\cite{VD83} is more adequate in the case of singular pencils. We say that a subspace ${\cal M}$ is a \emph{reducing subspace}
for the pencil $A-\lambda B$ if $\dim(A{\cal M}+B{\cal M})=\dim({\cal M})-k$,
where $k$ is the number of right singular blocks. In the example above, the reducing subspace
associated with the eigenvalue $\lambda_0=1$ is exactly given by all vectors $x(\alpha,\beta)$ with $\alpha,\beta\in\mathbb C$.
The \emph{minimal reducing subspace} ${\cal M}_{\text{RS}}(A,B)$ is the intersection of all reducing subspaces.
It is unique and can be numerically computed in a stable way from the generalized upper
triangular form (Guptri); see, e.g., \cite{DKa93}. Guptri exists for every pencil $A-\lambda B$ and has the form
\[P^*(A- \lambda B)Q=
\left[\begin{array}{ccc} A_{\rm r}-\lambda B_{\rm r} & \times & \times \\
0 & A_{\rm reg}-\lambda B_{\rm reg} & \times \\
0 & 0 & A_l-\lambda B_l\end{array}\right],
\]
where the matrices $P$ and $Q$ are unitary,
$A_{\rm r}-\lambda B_{\rm r}$ has only right singular blocks in its KCF,
$A_l-\lambda B_l$ has only left singular blocks in its KCF,
and $A_{\rm reg}-\lambda B_{\rm reg}$ has only regular blocks in its KCF.
We will briefly come back to a minimal reducing subspace in Section~\ref{sub:sing2EP}.

\section{Motivation and applications}
\label{sec:appl}
Before proposing our new method, we first review some motivating applications
where one is interested in computing the eigenvalues of a singular pencil,
besides the computation of transmission zeros already mentioned in the introduction.

\subsection{Differential algebraic equations and descriptor systems}

\label{sec:dae}
Linear differential algebraic equations (DAEs) with constant coefficients have the general
form
\[
E\dot x=Ax+f(t),\quad x(t_0)=x_0,
\]
where $E,A\in\mathbb R^{k,n}$, $t_0\in\mathbb R$, $x_0\in\mathbb R^n$, and $f:[t_0,\infty)\to\mathbb R^k$
is a given inhomogeneity; see \cite{KunM06}.
Linear time-invariant descriptor systems consist of a DAE combined with a system
input and output and take the form
\begin{align*}
E\dot x & = Ax+Bu,\quad x(t_0)=x_0,\\
y & = Cx+Du,
\end{align*}
where, in addition, $B\in\mathbb R^{k,m}$, $C\in\mathbb R^{p,n}$, $D\in\mathbb R^{p,m}$. Here,
$x:[t_0,\infty)\to\mathbb R^n$ stands for the state of the descriptor system, $u:[t_0,\infty)\to\mathbb R^m$ is the input
and $y:[t_0,\infty)\to\mathbb R^p$ the output. As highlighted in \cite{BenLMV15},
the problem may be well-posed even if the underlying pencil $A-\lambda E$ is singular. Indeed, even in
the singular case the corresponding DAE may have a  solution (even unique) for particular inhomogeneities
$f$
and special initial conditions $(t_0,x_0)$.

\subsection{Double eigenvalue problem}
\label{sec:double}
Given two $n\times n$ matrices $A$ and $B$, we are interested in all values $\lambda$ such
that $A + \lambda B$ has a double eigenvalue.
For the generic case of a double eigenvalue,
we look for independent vectors $x$ and $y$ such that
\begin{align*}
(A + \lambda B - \mu I) \, x & = 0, \\
(A + \lambda B - \mu I)^2 \, y & = 0,
\end{align*}
for some $\mu$.
This application is discussed in \cite{MPl14}, together with a staircase
algorithm to solve it; see also \cite{JKM11}. In the generic case, the problem has $n(n-1)$ solutions.

In \cite{MPl14}, the problem is solved by a linearization of the second equation
and followed by the solution of the obtained singular two-parameter eigenvalue problem.
The wanted values $\lambda$ are eigenvalues of
the singular pencil $\Delta_1-\lambda \Delta_0$ of size $3n^2\times 3n^2$, where
\[
\Delta_1 = A\otimes R - I\otimes P,\quad \Delta_0 = B\otimes R - I\otimes Q,
\]
for
\[P=\left[\begin{array}{ccc}A^2 & AB+BA & -2A \\0 & I & 0\\0 & 0 & I\end{array}\right],\
  Q=\left[\begin{array}{ccc}\phantom-0 & B^2 & -B \\-I & 0 & \phantom-0\\\phantom-0 & 0 & \phantom-0\end{array}\right],\
  R=\left[\begin{array}{ccc}\phantom-0 & -B & I \\\phantom-0 & \phantom-0 & 0\\-I & \phantom-0 & 0\end{array}\right].
\]
Generically, the pencil $\Delta_1-\lambda\Delta_0$ will have normal rank $3n^2-n$.
For more details, see \cite{MPl14} as well as \cite{HMP12} for other possible
linearizations.

\subsection{Singular two-parameter eigenvalue problems}\label{sub:sing2EP}
In the singular two-parameter eigenvalue problem we are given a pair of singular pencils
$\Delta_1-\lambda \Delta_0$ and $\Delta_2-\mu \Delta_0$ (see also \eqref{D012} and \eqref{Dgep}),
and the goal is to find \emph{finite regular eigenvalues} $(\lambda_0,\mu_0)$, where
(see, e.g., \cite{MPl09} for more details):
\begin{itemize}
\item $\lambda_0$ is a finite eigenvalue of $\Delta_1-\lambda \Delta_0$,
\item $\mu_0$ is a finite eigenvalue of $\Delta_2-\mu \Delta_0$,
\item and there exists a nonzero vector $z$ such that
$(\Delta_1-\lambda_0\Delta_0)z=0$, $(\Delta_2-\mu_0\Delta_0)z=0$, and
$z\not\in {\cal M}_{\text{RS}}(\Delta_i,\Delta_0)$ for $i=1,2$.
\end{itemize}
This problem requires more than just solving one singular GEP.
We discuss it in more details in Section~\ref{sec:2EP} and present a new numerical method for its solution.

%\subsection{Other sources}
%\label{sec:other}
%Also, some of the Matrix Market \cite{MM} pencils are (nearly) singular,
%such as the {\sf mhd} matrices arising in magnetohydrodynamics.

\section{Rank-completing perturbations of singular pencils}
\label{sec:rankcomplete}
Let $A-\lambda B$ be a singular pencil, where $A,B\in\mathbb C^{n,n}$
and $\textrm{nrank}(A,B)=n-k$. In this section
we investigate the effect of \emph{rank-completing perturbations}, i.e., rank-$k$
generic perturbations of the form
\begin{equation}
\label{pert1awithtau}
\widetilde A-\lambda \widetilde B := A-\lambda B+\tau \, (UD_AV^*-\lambda \, UD_BV^*),
\end{equation}
where $D_A,D_B\in\mathbb C^{k,k}$ are diagonal matrices such that $D_A-\lambda D_B$ is a regular
pencil, $U,V\in \mathbb C^{n,k}$ have full column rank, and $\tau\in\mathbb C$ is nonzero.
%A bit more general, we will consider perturbations of the form
%\begin{equation}
%\label{pert1awithtau}
%\widetilde A-\lambda \widetilde B := A-\lambda B+\tau \, (UD_AV^*-\lambda \, UD_BV^*),
%\end{equation}
%where $U,V,D_A,D_B$ are as above and $\tau\in\mathbb R$ is a real nonzero parameter,
We investigate
the dependence of eigenvalues and eigenvectors of the perturbed pencils on $\tau$.

Above and in the following, the
term generic is understood in the following sense. A set $\mathcal A\subseteq\mathbb C^m$ is called \emph{algebraic}
if it is the set of common zeros of finitely many complex polynomials in $m$ variables, and $\mathcal A$ is
called \emph{proper} if $\mathcal A\neq\mathbb C^m$. A set $\Omega\subseteq\mathbb C^m$ is called
\emph{generic} if its complement is contained in a proper algebraic set.
%If $(A,B)$ is an $n\times n$ pencil then we will consider perturbations of the form
%\begin{equation}
%\label{pert1a}
%A+UD_AV^*-\lambda \, (B+UD_BV^*),
%\end{equation}
%where $D_A,D_B\in\mathbb C^{k,k}$ are diagonal matrices such that $(D_A,D_B)$ is regular, and
%where $U,V\in C^{n,k}$ have full column rank.
In this sense we say that a property $\mathcal P$ holds generically with respect to the entries
of $U$ and $V$,
if there exists a generic set $\Omega\subseteq(\mathbb C^{n,k})^2$
(where we interpret $(\mathbb C^{n,k})^2$
as $\mathbb C^{2nk}$) such that $\mathcal P$ holds for all pencils of the form~\eqref{pert1awithtau} with
$(U,V)\in\Omega$.
%\medskip

\begin{remark}\rm Perturbations of the form~\eqref{pert1awithtau} are not the most general perturbations of
rank $k$. Indeed, completing $U$ and $V$ to nonsingular matrices $P=[U \ \, \widetilde U]$
and $Q=[V \ \, \widetilde V]$, we obtain that
\[
P^{-1}(UD_AV^*-\lambda\,UD_BV^*)(Q^*)^{-1}=\left[\begin{array}{cc}D_A-\lambda D_B&0\\ 0&0\end{array}\right],
\]
which means that the perturbation pencil $UD_AV^*-\lambda\,UD_BV^*$ has the regular part $D_A-\lambda D_B$ of
size $k\times k$, while, generically, a matrix pencil of nrank $k<n$ would have no regular part \cite{DD08}.
Thus, more generally one could consider perturbations of the form
\begin{equation}
\label{pert2a}
(A + U_1V^*_1+V_2U_2^*, \ B + U_1W^*_1+W_2U_2^*),
\end{equation}
where
$U_1,V_1,W_1\in\mathbb C^{n,\ell}$ and $U_2,V_2,W_2\in\mathbb C^{n,k-\ell}$ have full column rank for
$\ell\in\{0,1,\dots,k\}$.
However, we will restrict ourselves to perturbations of the form~\eqref{pert1awithtau}, because of their
favorable properties.
\end{remark}
%\medskip

Generically, a rank completing perturbation \eqref{pert1awithtau} of $A-\lambda B$ will result in a regular perturbed pencil
$\widetilde A-\lambda\widetilde B$.
%of full
%normal rank, i.e., the perturbed pencil is regular.
We will show in the following, that %if
%rank-completing perturbations of the form
%$\widetilde A-\lambda\widetilde B:=A-\lambda B+\tau \, (UD_AV^*-\lambda UD_BV^*)$ as in~
%\eqref{pert1awithtau}
%are considered, then
generically the KCF of $\widetilde A-\lambda\widetilde B$ is
for all $\tau\ne 0$ given by
\[
\left[\begin{array}{ccc}R_{\text{reg}}(\lambda) &0&0\\
0& R_{\text{pre}}(\lambda)&0\\ 0&0& R_{\text{ran}}(\lambda)\end{array}\right],
\]
where $R_{\text{reg}}(\lambda)$ is the regular part of the original pencil $A-\lambda B$,
$R_{\text{pre}}(\lambda)=D_A-\lambda D_B$ and $R_{\text{ran}}(\lambda)$ only has simple
eigenvalues that are different from the eigenvalues of $R_{\text{reg}}(\lambda)$ and $R_{\text{pre}}(\lambda)$.
Thus, the eigenvalues of $\widetilde A-\lambda\widetilde B$ are exactly the (say) $p$ eigenvalues
of the original pencil $A-\lambda B$ (counted with multiplicities) and $n-p$ newly generated eigenvalues which consist
of $k$ ``prescribed'' eigenvalues which are the eigenvalues of the perturbation pencil $U(D_A-\lambda D_B)V^*$,
and $n-p-k$ ``random'' eigenvalues that are gathered in the part $R_{\text{ran}}(\lambda)$.

We start by showing that under a rank-completing perturbation, the regular part of $A-\lambda B$
will stay invariant in the above sense.
%For this, the following auxiliary result is needed.
Although our main focus are square pencils, we will state some results in more generality covering also
the case of rectangular pencils. The following proposition is a generalization of \cite[Thm.~4.2]{MehMW16}
(which deals with Hermitian pencils) to a block case without any specific structure in $A$ and $B$.

\begin{proposition}\label{prop:genpert}
Let $A-\lambda B$ be an $n\times m$ singular matrix pencil having at least $k$ left minimal indices,
and let $U\in\mathbb C^{n,k}$. Then generically (with respect to the
entries of $U$) there exist nonsingular matrices $P,Q$ such that
\[
P(A-\lambda B)Q=\left[\begin{array}{cc}R(\lambda)&0\\ 0& S(\lambda)\end{array}\right]
\quad\mbox{and}\quad
PU=\left[\begin{array}{c}0\\ \widetilde U\end{array}\right],
\]
where $R(\lambda)$ and $S(\lambda)$ are the regular and singular parts of $A-\lambda B$,
respectively, and $PU$ is partitioned conformably with $P(A-\lambda B)Q$.
\end{proposition}

\begin{proof}
Without loss of generality we may assume that $A-\lambda B$ is already in the KCF
\[
A-\lambda B=\left[\begin{array}{cc} R(\lambda)&0\\ 0& S(\lambda)\end{array}\right],\quad
R(\lambda)=\left[\begin{array}{cc}J-\lambda I&0\\ 0&I-\lambda N\end{array}\right],
\]
where $S(\lambda)$ is the singular part and $R(\lambda)$ is the regular part of
$A-\lambda B$ with $J,N$ in Jordan normal canonical form and, in addition, $N$ nilpotent.
%The proof then uses similar ideas as the one of Theorem~4.2 in \cite{MehMW16}. Roughly speaking, its
Our main strategy is to use the part of $P$ that corresponds to $k$ arbitrarily chosen
left minimal indices to introduce zeros in the components of $U$ that correspond to the regular part of
$A-\lambda B$. Here, we can treat the components of $U$ corresponding to the ``finite eigenvalue part''
$J-\lambda I$ and the ``infinite eigenvalue part'' $I-\lambda N$ separately and we will give the
proof only for the first case as the proof for the ``infinite eigenvalue part'' is completely analogous.
Since we will only transform parts of the singular part $S(\lambda)$ that correspond to $k$
arbitrarily chosen left minimal indices and leave all other parts of $S(\lambda)$ unchanged,
it is sufficient to assume that $S(\lambda)$ consists of only $k$ singular blocks corresponding
to $k$ left minimal indices and that $R(\lambda)$ does not have a part corresponding
to infinite eigenvalues. This assumption will simplify the notation considerably.

Thus, we assume that $U$ and $A-\lambda B$ have the forms
\[
U=\bmat{ & \scriptstyle k\cr \scriptstyle r & U_0\cr\scriptstyle n_1+1 & U_1\cr
\vdots & \vdots\cr \scriptstyle n_k+1 & U_k}
\quad\mbox{and}\quad
A-\lambda B=\bmat{ & \scriptstyle r & \scriptstyle n_1 & \dots & \scriptstyle n_k\cr
\scriptstyle r & J-\lambda I & & & \cr
\scriptstyle n_1+1 &  &L_{n_1}(\lambda)^\top & & \cr
\vdots & & & \ddots & \cr
\scriptstyle n_k+1 & & & & L_{n_k}(\lambda)^\top}.
\]
In the following, we will use the notation $G_\ell=[0 \ \, I_\ell]^\top$ and
$H_\ell=[I_\ell \ \, 0]^\top$ to write
$L_\ell(\lambda)^\top=G_\ell-\lambda H_\ell$ for a singular block corresponding to a left minimal
index $\ell$. For the transformation matrices $P,Q$ we make the ansatz
\[
P=\bmat{ &\scriptstyle r &\scriptstyle n_1+1 & \dots &\scriptstyle n_k+1\cr
\scriptstyle r & I & P_1 & \dots & P_k\cr \scriptstyle n_1+1 & & I & &\cr
\vdots & & & \ddots & \cr \scriptstyle n_k+1 & & & & I}\quad\mbox{and}\quad
Q=\bmat{ &\scriptstyle r &\scriptstyle n_1 & \dots &\scriptstyle n_k\cr
\scriptstyle r & I & -P_1H_{n_1} & \dots & -P_kH_{n_k}\cr \scriptstyle n_1 & & I & &\cr
\vdots & & & \ddots & \cr \scriptstyle n_k & & & & I}
\]
As $P$ and $Q$ are designed in such a way that $B$ remains unchanged, we get
\[
P(A-\lambda B)Q =
\left[\begin{array}{cccc} J-\lambda I&P_1G_{n_1}-JP_1H_{n_1}& \dots & P_kG_{n_k}-JP_kH_{n_k}\\[0.5mm]
 &L_{n_1}(\lambda)^\top & & \\[-1mm] & & \ddots & \\ & & &L_{n_k}(\lambda)^\top\end{array}\right].
\]
It remains to find solutions $P_i$ to the equations $P_iG_{n_i}-JP_iH_{n_i}=0$
for $i=1,\dots,k$ to obtain $P(A-\lambda B)Q=A-\lambda B$.
Setting $P_i=[p_{i,0} \ \, p_{i,1} \ \, \dots \ \, p_{i, n_i}]$ for
$i=1,\dots,k$, the equations to be solved take the form
\[
[p_{i,1} \ \, \dots \ \, p_{i, n_i}]=P_iG_{n_i}=JP_iH_{n_i}
=[Jp_{i,0} \ \, \dots \ \, Jp_{i, n_i-1}],\quad i=1,\dots,k.
\]
Thus, we may choose $P_i=[p_{i,0} \ \, Jp_{i,0} \ \, \dots \ \, J^{n_i}p_{i,0}]$,
where $p_{i,0}\in\mathbb C^{r}$ is arbitrary. We will now use the freedom in the choice
of $p_{i,0}$ to guarantee that $PU$ has the desired form. To this end, let
%\[
%U_i=\left[\begin{array}{c}u_{i,0}^\top\\ u_{i,1}^\top\\\vdots\\ u_{i, n_i}^\top\end{array}\right],
%\]
%where
$u_{i,0}^\top,\dots,u_{i, n_i}^\top\in\mathbb C^k$ be the rows of $U_i$, $i=1,\dots,k$.
Then the first block component of $PU$ is given by $U_0+P_1U_1+\cdots+P_kU_k$ and to make it zero,
%set this component to zero,
we have to solve the equation
\begin{equation}\label{eq:pio}
-U_0=\sum_{i=1}^kP_iU_i=\sum_{i=1}^k\sum_{j=0}^{n_i}J^jp_{i,0}\,u_{i,j}^\top
\end{equation}
for $p_{1,0},\dots,p_{k,0}$. Using the $\text{vec}$-operation that ``vectorizes'' a matrix
by stacking the columns on top of each other and recalling the well-known identity
$\text{vec}(XYZ)=(Z^\top\otimes X)\,\text{vec}(Y)$
for matrices $X,Y,Z$, where $\otimes$ denotes the Kronecker product, we obtain that
\begin{equation}\label{eq:M}
-\text{vec}(U_0)=\sum_{i=1}^k\sum_{j=0}^{n_i}\big(u_{i,j}\otimes J^j\big)p_{i,0}
=M\cdot \bigg[ \begin{array}{c} p_{1,0}\\[-1.5mm] \vdots\\[-1.5mm] p_{k,0} \end{array} \bigg],
%=M\cdot\left[\begin{array}{c}p_{1,0}\\[-1.5mm] \vdots\\[-1.5mm] p_{k,0}\end{array}\right],
\end{equation}
where
\[
M=\Big[\sum_{j=0}^{n_1}u_{1,j}\otimes J^j \ \, \dots \ \, \sum_{j=0}^{n_k}u_{k,j}\otimes J^j\Big] \in\mathbb C^{rk,rk}.
\]
The determinant of $M$ is a polynomial in the $nk$ entries of $U$ (in fact, it only depends on
the entries of $U_1,\dots,U_k$) which is nonzero for the particular choice $u_{1,0}=e_1,\dots,u_{k,0}=e_k$
and $u_{i,j}=0$ for $j>0$, where $e_1,\dots,e_k$ denote the standard basis vectors of $\mathbb C^k$.
(Indeed, in this case $M$ is just the identity of size $rk\times rk$.) Thus, generically (with respect to
the entries of $U$), the matrix $M$ is invertible, so equation~\eqref{eq:M} and thus also~\eqref{eq:pio}
can be uniquely solved for $p_{1,0},\dots,p_{k,0}$ which finishes the proof.
\end{proof}

The next result shows that the canonical form of the regular part of the original pencil stays invariant
under a generic rank-completing perturbation of the form~\eqref{pert1awithtau}.
Concerning the eigenvalues of the perturbed pencil that are also eigenvalues of the original
singular pencil, the result also states that the corresponding left and right eigenvectors satisfy a
particular orthogonality relation.

\begin{theorem}\label{thm:nfp}
Let $A-\lambda B$ be an $n\times n$ singular pencil of normal rank $n-k$, let $U,V\in\mathbb C^{n,k}$ have full column rank
and let $D_A,D_B\in\mathbb C^{k,k}$ be such that $D_A-\lambda D_B$ is regular and all
eigenvalues of $D_A-\lambda D_B$ are distinct from the eigenvalues of $A-\lambda B$.
%Furthermore, let
%\begin{equation}\label{eq:pertUVnew}
%\widetilde A-\lambda\widetilde B:=A-\lambda B+\tau \, (UD_AV^*-\lambda UD_BV^*),
%\end{equation}
%where $\tau\in\mathbb R$.
Then,
generically with respect to the entries of $U$ and $V^*$, the following
statements hold for the pencil \eqref{pert1awithtau}:
\begin{enumerate}
\item For each $\tau\neq 0$, there exist nonsingular matrices $\widetilde P$ and $\widetilde Q$ such that
\begin{equation}\label{eq:nfp}
\widetilde P\big(\widetilde A-\lambda\widetilde B\big)\widetilde Q=\left[\begin{array}{cc}R(\lambda)&0\\
0& R_{\rm new}(\lambda)\end{array}\right],
\end{equation}
where $R(\lambda)$ is the regular part of the original pencil $A-\lambda B$,
and $R_{\rm new}(\lambda)$ is regular and all
its eigenvalues are distinct from the eigenvalues of
$R(\lambda)$.
\item If $\lambda_0$ is a finite eigenvalue of $A-\lambda B$, i.e., $\text{rank}(A-\lambda_0 B)<n-k$,
then $\lambda_0$ is an eigenvalue of~\eqref{pert1awithtau} for each $\tau\ne 0$. Furthermore, the right null space
    ${\cal N}_r(\lambda_0):=\ker\big(\widetilde A-\lambda_0 \widetilde B\big)$ and
    the left null space ${\cal N}_l(\lambda_0):=
    \ker\big(\big(\widetilde A-\lambda_0 \widetilde B\big)^*\big)$ are both
    constant in $\tau\ne 0$. In addition:
    \begin{enumerate}
    \item ${\cal N}_r(\lambda_0)\perp {\rm span}(V)$, i.e., if $x$ is a right eigenvector of
    $\widetilde A-\lambda\widetilde B$ associated with $\lambda_0$, then $V^*x=0$.
    \item ${\cal N}_l(\lambda_0)\perp {\rm span}(U)$, i.e., if $y$ is a left eigenvector of
    $\widetilde A-\lambda\widetilde B$ associated with  $\lambda_0$, then $U^*y=0$.
    \end{enumerate}
\item If $\infty$ is a eigenvalue of $A-\lambda B$, i.e., $\text{rank}(B)<n-k$,
then $\infty$ is an eigenvalue of~\eqref{pert1awithtau} for each $\tau\ne 0$. The right and left null spaces
    ${\cal N}_r(\infty):=\ker(\widetilde B)$ and
    ${\cal N}_l(\infty):=\ker\big(\widetilde B^*\big)$ are both
    constant in $\tau\ne 0$. In addition:
    \begin{enumerate}
    \item ${\cal N}_r(\infty)\perp {\rm span}(V)$, i.e., if $x$ is a right eigenvector of
    $\widetilde A-\lambda\widetilde B$ associated with $\infty$, then $V^*x=0$.
    \item ${\cal N}_l(\infty)\perp {\rm span}(U)$, i.e., if $y$ is a left eigenvector of
    $\widetilde A-\lambda\widetilde B$ associated with $\infty$, then $U^*y=0$.
    \end{enumerate}
\end{enumerate}
\end{theorem}

\begin{proof}
First, we will assume that $A-\lambda B$ does not have one of
the eigenvalues $0$ or $\infty$ and we will show 1) and 2) for this particular case.

Applying Proposition~\ref{prop:genpert}, there generically exist nonsingular matrices $P,Q$ such that
\begin{equation}\label{eq:help1}
P(A-\lambda B)Q=\left[\begin{array}{cc}R(\lambda)&0\\ 0& S(\lambda)\end{array}\right],\quad
PU=\left[\begin{array}{c}0\\ U_2\end{array}\right],\quad Q^*V=\left[\begin{array}{c}V_1\\ V_2\end{array}\right],
\end{equation}
where $R(\lambda)$ and $S(\lambda)$ are the regular and singular parts of $A-\lambda B$,
respectively, both being in KCF, and
where $U$ and $V$ are partitioned conformably with $A-\lambda B$. We will now show 1) and 2):

1) Since $A-\lambda B$ is square and of normal rank $n-k$, it has exactly $k$ left minimal indices, say $n_1,\dots, n_k$
and exactly $k$ right minimal indices, say $m_1,\dots,m_k$. We may assume without loss of generality
that they are paired up to form square blocks of one left and right minimal index each, i.e., we may
assume that $S(\lambda)$ has the block diagonal form
\[
S(\lambda)=\text{diag}\left(\left[\begin{array}{cc}L_{n_1}(\lambda)^\top &0\\ 0&L_{m_1}(\lambda)\end{array}\right],\dots,
\left[\begin{array}{cc}L_{n_k}(\lambda)^\top &0\\ 0&L_{m_k}(\lambda)\end{array}\right]\right).
\]
Then the perturbed pencil takes the form
\[
P(\widetilde A-\lambda \widetilde B)Q=\left[\begin{array}{cc}R(\lambda) & 0\\
\tau \, U_2(D_A-\lambda D_B)V_1^* & R_{\rm new}(\lambda)\end{array}\right],
\]
where
\[
R_{\rm new}(\lambda):= S(\lambda)+\tau \, U_2(D_A-\lambda D_B)V_2^*.
\]
Clearly, the determinant of $P(\widetilde A-\lambda \widetilde B)Q$ is
equal to $\det R(\lambda)\cdot\det R_{\rm new}(\lambda)$
and from the definition of $R_{\rm new}(\lambda)$ it is clear that the coefficients of
$\det R_{\rm new}(\lambda)$ are polynomials in the entries of $U_2$ and $V_2^*$ and thus also
of $U$ and $V^*$.

Now let $\lambda_0$ be an eigenvalue of $R(\lambda)$, i.e., $\det R(\lambda_0)=0$.
Note that if $e_{j,\ell}$ denotes the $j$th standard basis vector of $\mathbb C^{\ell}$,
and $F_i = (\alpha_i-\lambda_0\beta_i)e_{n_i+1,n_i+m_i+1}e_{n_i+1,n_i+m_i+1}^*$, then
\[
\det\left(\left[\begin{array}{cc}L_{n_i}(\lambda_0)^\top &0\\ 0&L_{m_i}(\lambda_0)\end{array}\right]+F_i\right)
=(-\lambda_0)^{n_i}(\alpha_i-\lambda_0\beta_i).
\]
Thus, with $e_j$ the $j$th standard basis vector in $\mathbb C^n$, for the particular choice
\[
U_2=V_2=[e_{n_1+1} \ \, e_{n_1+m_1+1+ n_2+1} \ \, \dots \ \,
e_{n_1+m_1+1+\cdots+n_{k-1}+m_{k-1}+1+n_k+1}]
\]
we obtain that
\[
\det R_{\rm new}(\lambda_0)= \tau^k \, (-\lambda_0)^{n_1+\cdots+ n_k}
(\alpha_1-\lambda_0\beta_1) \cdots (\alpha_k-\lambda_0\beta_k)
%\prod_{i=1}^k(\alpha_i-\lambda_0\beta_i)
\]
which is nonzero as the eigenvalues of $D_A-\lambda D_B$ are by hypothesis distinct from $\lambda_0$.
But then $\det R_{\rm new}(\lambda_0)$ is generically nonzero (the set of all $(U,V^*)$
for which $\det R_{\rm new}(\lambda_0)=0$ is by definition an algebraic set, because
$\det R_{\rm new}(\lambda_0)$ is a polynomial in the entries of $U$ and $V^*$) which shows that
$R_{\rm new}(\lambda_0)$ is generically regular and does not have $\lambda_0$ as an eigenvalue. Since intersections
of finitely many generic sets are still generic we can conclude that the spectra of $R(\lambda)$
and $R_{\rm new}(\lambda)$ are disjoint. But then it immediately follows from \cite[Lemma 6.11]{MacMMM13}
and \cite[XII.2, Thm.~2]{Gan59} that the perturbed pencil $\widetilde A-\lambda \widetilde B$ has
the KCF as given in~\eqref{eq:nfp}.

2) Let $\lambda_0$ be a eigenvalue of $A-\lambda B$ and thus of $R(\lambda)$. It then follows
directly from 1) that $\lambda_0$ is also an eigenvalue of $\widetilde A-\lambda\widetilde B$
for each $\tau\neq 0$. For the moment, let $\tau\neq 0$ be fixed and let
the columns of $Y$ form a basis of the left null space $\mathcal N_l(\lambda_0)$ of $\widetilde A-\lambda
\widetilde B$. Partition
\[
Y^*P^{-1}=[Y_1^* \ \, Y_2^*]
\]
conformably with the partition in~\eqref{eq:help1}. Since $\lambda_0$ is not an eigenvalue of
$R_{\rm new}(\lambda)$ we obtain from $Y^*(\widetilde A-\lambda_0 \widetilde B)=0$ that $Y_2=0$.
But this implies that the columns of $Y$ form a basis for the left null space $\mathcal N_l(\lambda_0)$
for all values $\tau\ne 0$ as the construction of the transformation matrices
$P$ and $Q$ only depends on $A$, $B$, and $U$, but not on $\tau$. Furthermore, we obtain
\[
Y^*U=Y^*P^{-1}PU = [Y_1^* \ \, 0]
\cdot\left[\begin{array}{c}0\\ U_2\end{array}\right]=0,
\]
i.e., $N_l(\lambda_0)$ is orthogonal to the space spanned by the columns of $U$.

Observe that the statement on the right null space $\mathcal N_r(\lambda_0)$ does not follow
immediately from the partitioning in~\eqref{eq:help1} as in general we have $V_1\neq 0$.
But we can apply the already proved part of the theorem to the
pencil $A^*-\lambda B^*$ and the perturbation $V(D_A^*-\lambda D_B^*)U^*$ to obtain the corresponding
statements for the right null space $N_r(\lambda_0)$.
This finishes the proof of 2).

Finally, assume that $A-\lambda B$ does have one of the eigenvalues $0$ or $\infty$. Then
apply a M\"obius transformation of the form
\[
\mathrm M_{\alpha,\beta}(A-\lambda B):=\alpha A+\beta B-\lambda \, (\alpha B-\beta A)
\]
where $\alpha,\beta\in\mathbb R$ are such that $\alpha^2+\beta^2=1$ and such that
$\mathrm M_{\alpha,\beta}(A-\lambda B)$ does neither have the eigenvalues $0$ nor $\infty$.
Note that this M\"obius transformation just has the effect of ``rotating'' eigenvalues
on the extended real line $\mathbb R\cup\{\infty\}$, but it leaves eigenvectors and the Jordan
structure invariant, see, e.g., \cite{MacMMM15}. The result then follows by applying the already
proved parts of the theorem on $\mathrm M_{\alpha,\beta}(A-\lambda B)$ followed by applying the inverse
M\"obius transformation
\[
\mathrm M_{\alpha,-\beta}(C-\lambda D):=\alpha C-\beta D-\lambda \, (\alpha D+\beta C).
\]
to give the corresponding statements for $A-\lambda B$.
In particular, this shows 3).
\end{proof}

\begin{remark}\rm
We mention that part 1) in Theorem~\ref{thm:nfp} is in line with one of the main results of \cite{DD07},
where it was shown that generically the regular part of a singular pencil stays invariant under
generic perturbations that do not make the pencil regular. Part 1) of Theorem~\ref{thm:nfp} extends this
result (in the sense of the theorem) to the case of rank-completing perturbation. Clearly, the regular
part of the pencil will be completely changed if generic perturbations of a rank larger than the
difference of the size and the normal rank of the pencil are applied.
\end{remark}

Theorem~\ref{thm:nfp} characterizes the properties of the eigenvalues from the block $R$ of the perturbed
pencil $\widetilde A-\lambda \widetilde B$ as in~\eqref{pert1awithtau}, i.e., of the eigenvalues
that coincide with the eigenvalues of the unperturbed pencil. We will next
investigate the properties of the eigenvalues from the newly created block $R_{\rm new}$.
We start with the following lemma that will be needed for the main results. The values
$\gamma_1,\ldots,\gamma_k$ in the lemma are the eigenvalues that we will prescribe later in
Theorem~\ref{thm:mainBor} using the matrices $D_A$ and $D_B$.

\begin{lemma}\label{lem:Lo} Let $A-\lambda B$ be an $n\times n$ singular pencil of normal rank $n-k$
with left minimal indices $n_1,\dots,n_k$ and right minimal indices $m_1,\dots,m_k$. Furthermore, let
$U,V\in\mathbb C^{n,k}$ have full column rank, $N:=n_1+\cdots+n_k$, $M:=m_1+\cdots+m_k$, and let
$\gamma_1,\dots,\gamma_k\in\mathbb C$
be given values that are distinct from the eigenvalues of $A-\lambda B$.
Then, generically with respect to the entries of $U$ and $V^*$, the following statements hold:
\begin{enumerate}
\item
There exist exactly $M$ pairwise distinct values $\alpha_1,\ldots,\alpha_M$
different from the eigenvalues of $A-\lambda B$ and different
from $\gamma_1,\ldots,\gamma_k$
such that for each
$\alpha_i$ there exists a nonzero vector $z_i$ with $(A-\alpha_iB)z_i=0$ and $V^*z_i=0$.
\item There exist exactly $N$ pairwise distinct values $\beta_1,\ldots,\beta_N$ different from the
eigenvalues of $A-\lambda B$ and different from $\gamma_1,\ldots,\gamma_k$ and $\alpha_1,\dots,\alpha_M$ such that for each $\beta_i$ there
exists a nonzero vector $w_i$ with $w_i^*(A-\beta_iB)=0$ and $w_i^*U=0$.
\item For any given set of
$k$ linearly independent vectors $t_1,\dots,t_k\in\mathbb C^k$ there exist nonzero vectors
$s_1,\dots,s_k$ with $(A-\gamma_i B)s_i=0$ and $t_i=V^*s_i$ for $i=1,\dots,k$.
\end{enumerate}
\end{lemma}

\begin{proof}
1) Without loss of generality we may assume that $A-\lambda B$ is in KCF such
that the blocks $L_{m_1}(\lambda),\dots,L_{m_k}(\lambda)$ associated with the right minimal indices
appear first in the form. Then for each $\alpha\in\mathbb C$ different from the eigenvalues of
$A-\lambda B$ the columns of
\[
\left[\begin{array}{ccc}q_1(\alpha)&\dots&q_k(\alpha)\end{array}\right]=\left[\begin{array}{ccc}q_{11}(\alpha)&&0\\[-2mm]
& \ddots & \\ 0&&q_{kk}(\alpha)\\\hline 0&\dots&0\end{array}\right]
\quad\mbox{with }\; q_{jj}(\alpha)=\left[\begin{array}{c}1\\[-0.5mm] \alpha\\[-1.5mm] \vdots\\ \alpha^{m_j}\end{array}\right]
\]
form a basis for $\ker(A-\alpha B)$.
(When $\alpha$ is an eigenvalue of $A-\lambda B$, there are additional vectors in $\ker(A-\alpha B)$ since the rank of
$A-\alpha B$ drops below the normal rank $n-k$.)

We are looking for $z\ne 0$ and $\alpha$ such that
$V^*z=0$ and $(A-\alpha B)z=0$. Since we want $\alpha$ to be distinct from the eigenvalues of
$A-\lambda B$, the vector $z$ has to be of the form
\[z=c_1\,q_1(\alpha)+\cdots+c_k\,q_k(\alpha),\]
where $c=[c_1 \ \, \ldots \ \, c_k]^T\ne 0$. From $V^*z=0$ we get the equation
\begin{equation}\label{eq:gc}
G(\alpha)\,c=0,
\end{equation}
where
$G(\alpha)$ is a $k\times k$ matrix whose element $g_{ij}(\alpha)=v_i^*q_j(\alpha)$ is a polynomial in
$\alpha$ which generically with respect to the entries of $v_i^*$ will have degree $m_j$ for $i,j=1,\ldots,k$.
Equation \eqref{eq:gc} has a nontrivial solution if and only if
$\det G(\alpha)=0$, where $\det G(\alpha)$ is a polynomial in $\alpha$ which generically with respect to
the entries of $V^*$ is of degree $M$. Thus $\det G(\alpha)$ will have $M$ roots $\alpha_1,\dots,\alpha_M$
(counted with multiplicities).

On the other hand, for each fixed $\mu\in\mathbb C$, we have that
$\det G(\mu)$ is also a polynomial in the entries of $V^*=[v_1 \ \, \dots \ \, v_k]^*$.
For the particular choice
\[
v_1=e_1, \quad v_2=e_{m_1+2}, \quad \ldots, \quad v_k=e_{m_1+\cdots+m_{k-1}+k}
\]
we obtain that $v_i^*q_j(\mu)=\delta_{ij}$ so that $G(\mu)=I_k$ shows that $\det G(\mu)$ is a nonzero
polynomial in the entries of $V^*$. It thus follows that generically with respect to the entries of $V^*$
we will have $\det G(\mu)\neq 0$, and consequently the fixed value $\mu$ will generically not be among
the roots of $G(\alpha)$ as a polynomial in $\alpha$. Since the intersection of finitely many generic
sets is still generic, it follows that we can generically exclude finitely many values from the
zeros $\alpha_1,\dots,\alpha_M$ of $G(\alpha)$. This shows that generically with respect to the entries of $V^*$,
the values $\alpha_1,\dots,\alpha_M$ are different from the eigenvalues of $A-\lambda B$
and also from the given values $\gamma_1,\dots,\gamma_k$.

Next we show that the roots $\alpha_1,\dots,\alpha_M$ of $p(\alpha):=\det G(\alpha)$ generically are
pairwise distinct. This is exactly the case if the discriminant $\operatorname{Disc}(p)$ of $p$ is nonzero. Since
$\operatorname{Disc}(p)$ is a polynomial in the entries of $p$ (this is well known, but can also be seen from the fact
that the discriminant is a scalar multiple of the determinant of the Sylvester matrix $S(p,p')$ associated
with $p$ and its formal derivative $p'$), it follows that $\operatorname{Disc}(p)$ is
a polynomial with respect to the entries of $V^*$. It remains to show that $\operatorname{Disc}(p)$ is a nonzero
polynomial (because then we will have that $\operatorname{Disc}(p)\neq 0$ is a generic property with respect to
the entries of $V^*$), and for this it is enough to show that for a particular choice of the entries of $V$
we have that the values $\alpha_1,\dots,\alpha_M$ are pairwise distinct. Now taking
$v_1=e_{m_1+1}-\varepsilon_1 e_1$, $v_2=e_{m_1+m_2+2}-\varepsilon_2 e_{m_1+2}$, \dots,
$v_k=e_{m_1+\cdots+m_{k}+k}-\varepsilon_ke_{m_1+\cdots+m_{k-1}+k}$,
with $\varepsilon_1,\dots,\varepsilon_k>0$, we obtain that $v_i^*q_j(\alpha)=\delta_{ij}\alpha^{m_j}-\varepsilon_j$
and thus $G(\alpha)$ is diagonal and
\[
\det G(\alpha)= (\alpha^{m_1}-\varepsilon_1) \cdots (\alpha^{m_k}-\varepsilon_k).
%\prod_{j=1}^k(\alpha^{m_j}-\varepsilon_j).
\]
Since the roots of each factor $(\alpha^{m_j}-\varepsilon_j)$ are $m_j$ pairwise distinct complex numbers
lying on a circle centered at zero with radius $\varepsilon_j^{1/m_j}$, it remains to choose the values
$\varepsilon_1,\dots,\varepsilon_k$ in such a way that the $k$ radii are pairwise distinct to
obtain the desired example.

2) In a way similar to the one in 1) we can consider the left null space for $A-\alpha B$ and show
the existence of $\beta_1,\ldots,\beta_N$ and the corresponding nonzero vectors $w_1,\ldots,w_N$, where
now the statements are generic with respect to the entries of $U$. In particular, by interpreting $V$
as already fixed, this shows that generically with respect to the entries of $U$, the values
$\beta_1,\dots,\beta_N$ are not only different from the eigenvalues of $A-\lambda B$ and
$\gamma_1,\dots,\gamma_k$, but also from the values $\alpha_1,\dots,\alpha_M$ constructed in $1)$.

3) With the same notation as in 1) we now aim to solve the equations
\[
s_i=c_1q_1(\gamma_i)+\cdots+c_kq_k(\gamma_i)\quad\mbox{and}\quad V^*s_i=t_i,
\]
or, equivalently, $G(\gamma_i)c=t_i$ for $i=1,\dots,k$. Since $\gamma_i$ is different from the
values $\alpha_1,\dots,\alpha_M$, we have $\det G(\gamma_i)\neq 0$ and hence $G(\gamma_i)\,c=t_i$
is uniquely solvable for $c$ for $i=1,\dots,k$.
\end{proof}
%\medskip

%\begin{remark}\label{rm:symmpert}\rm
%{(\bf CM: Not 100\% true: If the pencil is Hermitian, then the values $\alpha_i$ will be the
%conjugates of $\beta_i$, but I expect them generically to be complex. (This is what we observed
%for the rank-1 case in \cite{MehMW16}). If the pencil, however is real symmetric, then it will have complex
%and real eigenvalues and the real eigenvalues are indeed double.)}
%If we have a symmetric singular pencil $A-\lambda B$ then we have to
%use a nonsymmetric perturbation $UD_AV^*-\lambda UD_BV^*$. If not, the values $\alpha_i$
%and $\beta_j$ in Lemma \mbox{\rm\ref{lem:Lo}} cannot be different,
%see also Remark~\mbox{\rm\ref{rm:symmpert2}}.
%\end{remark}

%\medskip

The following theorem encapsulates the main result on the new eigenvalues of our perturbed pencil.

\begin{theorem}\label{thm:mainBor} Let $A-\lambda B$ be an $n\times n$ singular pencil of normal rank $n-k$
with left minimal indices $n_1,\dots,n_k$ and right minimal indices $m_1,\dots,m_k$.
Furthermore, let $U,V\in\mathbb C^{n,k}$ have full column rank
and let $D_A=\text{diag}(a_1,\dots,a_k),D_B=\text{diag}(b_1,\dots,b_k)\in\mathbb C^{k,k}$ be
such that $D_A-\lambda D_B$ is regular and such that all (not necessarily pairwise distinct) values
$\gamma_i:=\frac{a_i}{b_i}$, $i=1,\dots,k$,
are different from the eigenvalues of $A-\lambda B$. (Here, $\frac{a_i}{b_i}$ is interpreted as the
infinite eigenvalue, if $b_i=0$.) Finally, let $N:=n_1+\cdots+n_k$ and $M:=m_1+\cdots+m_k$.
%and let
%\begin{equation}\label{eq:pertUVnew2}
%\widetilde A-\lambda\widetilde B:=A-\lambda B+\tau \, (UD_AV^*-\lambda UD_BV^*),
%\end{equation}
%where $\tau\in\mathbb R$.
Then generically with respect to the entries of $U$ and $V^*$, the following
statements hold:
\begin{enumerate}
\item The pencil~\eqref{pert1awithtau} has $M$ simple eigenvalues $\alpha_1,\ldots,\alpha_M$
    which are independent of $\tau\ne 0$, so that
    for each of these eigenvalues its right eigenvector $x_i$ is constant in $\tau\ne 0$ (up to scaling) and satisfies
    $V^* x_i=0$, while the left eigenvector $y_i$ is a linear function of $\tau$ (up to scaling) and
    satisfies $U^*y_i\ne 0$ for all $\tau\ne 0$.
\item The pencil~\eqref{pert1awithtau} has $N$ simple eigenvalues $\beta_1,\ldots,\beta_N$
    which are independent of $\tau\ne 0$, so that
    for each of these eigenvalues its left eigenvector $y_i$ is constant for $\tau\ne 0$ (up to scaling) and
    satisfies $U^* y_i=0$, while the right eigenvector $x_i$ is a linear function of $\tau$ (up to scaling)
    and satisfies $V^*x_i\ne 0$ for all $\tau\ne 0$.
\item For each $\tau\neq 0$ each $\gamma_i$ is an eigenvalue of~\eqref{pert1awithtau} with the same algebraic
    multiplicity as for the pencil $D_A-\lambda D_B$. Furthermore, the left and right null spaces
    ${\cal N}_l(\gamma_i)$ and ${\cal N}_r(\gamma_i)$ of~\eqref{pert1awithtau} associated with $\gamma_i$
    are constant in $\tau$. In addition, we have:
    \begin{enumerate}
    \item ${\cal N}_r(\gamma_i)\cap \ker(V^*)=\{0\}$, i.e., for each right eigenvector $x$ of~\eqref{pert1awithtau}
    associated with $\gamma_i$ we have $V^*x\ne 0$.
    \item ${\cal N}_l(\gamma_i)\cap \ker(U^*)=\{0\}$, i.e., for each left eigenvector $y$ of ~\eqref{pert1awithtau}
    associated with $\gamma_i$ we have $U^*y\ne 0$.
    \end{enumerate}
\end{enumerate}
(Note that the simplicity of the eigenvalues $\alpha_1,\dots,\alpha_M,\beta_1,\dots,\beta_N$ implies
%in particular
that they are all different from the eigenvalues of $A-\lambda B$ and
% from % the eigenvalues
$\gamma_1,\dots,\gamma_k$.)
\end{theorem}

\begin{proof}
Without loss of generality, we may assume that the infinite eigenvalue is not among the eigenvalues of $D_A-\lambda D_B$.
Otherwise, we may as in the proof of Theorem~\ref{thm:nfp} apply a M\"obius
transformation to both $A-\lambda B$ and $D_A-\lambda D_B$ such that $D_A-\lambda D_B$ does not have
the eigenvalue $\infty$, apply the statement that was proved for this special situation, and finally
transform back with the inverse M\"obius transformation to obtain the desired result.

Observe that generically with respect to the entries of $U$ and $V^*$, the statements of Lemma~\ref{lem:Lo} %will
hold, if we take the eigenvalues
of the pencil $D_A-\lambda D_B$
for the values $\gamma_1,\dots,\gamma_k$ and take standard basis vectors $e_1,\dots,e_k$ from $\mathbb C^k$ for
the vectors $t_1,\dots,t_k$. We now show 1)--3).

1) By Lemma~\ref{lem:Lo} there exist exactly $M$ pairwise distinct values
$\alpha_1,\ldots,\alpha_M$ different from the eigenvalues of $A-\lambda B$
and from $\gamma_1,\dots,\gamma_k$, and nonzero vectors $z_1,\ldots,z_M$ such that
$(A-\alpha_i B)z_i=0$ and $V^*z_i=0$ for $i=1,\ldots,M$. From this we obtain that
$(A-\alpha_i B + \tau \, UD_AV^* -\alpha_i \tau \, UD_BV^*)\,z_i=0$ which means that
$\alpha_i$ is an eigenvalue of \eqref{pert1awithtau} for $\tau\ne 0$ with
a right eigenvector $z_i$ that is invariant under $\tau$.

Considering now $\tau$ as a variable, it follows that the pencil %${\rm rank}(A-\alpha_i B+\tau \, U(D_A-\alpha_iD_B)V^*)\le n-1$ for
%all $\tau\ne 0$, which means that
\begin{equation}\label{eq:pengh}
G_i+\tau H_i:=(A-\alpha_i B) +\tau \, U(D_A-\alpha_iD_B)V^*\end{equation}
is singular. Suppose that ${\rm nrank}(G_i,H_i)=n-j$ for $j\ge 1$, which means that \eqref{eq:pengh} has $j$ right and $j$ left minimal indices.
We know from $G_iz_i=0$ and $H_iz_i=0$ that one %of these
right minimal index is equal to zero.
The remaining $j-1$ right minimal indices are all larger than zero, because otherwise there would exist
$y_i\in \ker(G_i)\cap\ker(H_i)$ linearly independent of $z_i$ which implies that
$\alpha_i$ would be a multiple eigenvalue of \eqref{pert1awithtau} in contradiction to Lemma~\ref{lem:Lo}.

Now suppose that \eqref{eq:pengh} has a left minimal index being zero.
Then there exists a vector $w_i\neq 0$ such that $w_i^*G_i=0$ and $w_i^*H_i=0$, which implies $w_i^*U=0$,
because $V$ has %was assumed to have
full rank and $D_A-\alpha_i D_B$ is nonsingular since % the value
$\alpha_i$ is different from the eigenvalues of $D_A-\lambda D_B$. But then by Lemma~\ref{lem:Lo} $\alpha_i$
is equal to one of the values $\beta_1,\dots,\beta_N$ which is a contradiction. Thus, all left minimal
 indices of~\eqref{eq:pengh} are larger than or equal to one.

Furthermore, we know that $\text{rank}(G_i)=n-k$ and $\text{rank}(H_i)=k$ since
$\alpha_i$ differs from all finite eigenvalues of $A-\lambda B$ and all eigenvalues of $D_A-\lambda D_B$. It follows that in the KCF of
the pencil \eqref{eq:pengh} there are at least $n-k-j$ blocks associated with the eigenvalue infinity
and at least $k-j$ blocks associated with the eigenvalue zero.

By a simple computation we obtain that the dimension of the KCF of \eqref{eq:pengh}
is at least $(n+j-1)\times (n+j-1)$; therefore the
only option is $j=1$ and hence \eqref{eq:pengh} has exactly one right minimal index (being zero) and
exactly one left minimal index, say $p$. Then another simple computation shows that the dimension
of the KCF of \eqref{eq:pengh} is at least $(n+p-1)\times (n+p-1)$ showing that the left minimal
index $p$ must be equal to one.
%Assume that $p=0$. Then there would
%be a vector $w\neq 0$ such that $w^*(A-\alpha_i B)=0$ and $w^*U(D_A-\alpha_iD_B)V^*=0$, which implies $w^*U=0$,
%because $V$ was assumed to have full rank and $D_A-\alpha_i D_B$ is nonsingular since the values
%$\alpha_i$ are different from the eigenvalues of $D_A-\lambda D_B$. But then by Lemma~\ref{lem:Lo} $\alpha_i$
%is equal to one of the values $\beta_1,\dots,\beta_N$ which is a contradiction. Thus, we
%have $p>0$.
%Now, suppose that for a certain $\tau\ne 0$ there exists a left
%eigenvector $y_i$ of \eqref{pert1awithtau} associated with $\alpha_i$ such that $y_i^*U=0$. It follows
%that $y_i^*(A-\alpha_i B)=0$ (which would also imply that $y_i$ is constant in $\tau$)
%and thus the pencil $(A-\alpha_i B)+\lambda U(D_A-\alpha_i D_B)V^*$ has a left minimal index $0$,
%which is a contradiction. So, $y_i^*U\ne 0$ for all $\tau\ne 0$ and for each left eigenvector $y_i$
%of~\eqref{pert1awithtau} associated with $\alpha_i$.
 Consequently, there exist linearly independent vectors
$w_i$ and $z_i$ such that
\begin{align*}
w_i^*(A-\alpha_iB)&=0,\\
z_i^*(A-\alpha_iB)+w_i^*U(D_A-\alpha_i D_B)V^*&=0,\\
z_i^*U(D_A-\alpha_i D_B)V^*&=0
\end{align*}
and $w_i^*U\ne 0$. Up to scaling, the left eigenvector $y_i$ of~\eqref{pert1awithtau}
associated with $\alpha_i$ then has the form
$y_i(\tau)=w_i+\tau z_i$
and is a linear function of $\tau$.

2) This follows completely analogously to 1).

3) Clearly, the standard basis vectors $e_1,\dots,e_k$ are eigenvectors of the pencil $D_A-\lambda D_B$
associated with the eigenvalues $\gamma_1,\dots,\gamma_k$. By Lemma~\ref{lem:Lo}, there exist
$k$ (necessarily linearly independent) vectors $s_1,\dots,s_k\in\mathbb C^n$ such that
$(A-\gamma_i B)s_i=0$ and $e_i=V^*s_i$ for $i=1,\dots,k$. Then  we have
\[
(\widetilde A-\gamma_i\widetilde B)s_i=(A-\gamma_i B)s_i+\tau \, U(D_A-\gamma_i D_B)V^*s_i=0
\]
for each $\tau\neq 0$ for $i=1,\dots,k$.
This implies that the values $\gamma_1,\dots,\gamma_k$ are eigenvalues of $\widetilde A-\lambda\widetilde B$
with the same algebraic multiplicities as for $D_A-\lambda D_B$. Furthermore, it follows that the null space
${\cal N}_r(\gamma_i)$ does not depend on $\tau$ and by construction we have
${\cal N}_r(\gamma_i)\cap \ker(V^*)=\{0\}$.

By applying Lemma~\ref{lem:Lo} to the pencil $A^*-\lambda B^*$ we obtain the analogous statements for
the left null spaces ${\cal N}_l(\gamma_i)$.
\end{proof}
%\bigskip

\begin{summary} \rm\label{rm:3groups}
Summarizing the results from Theorem~\ref{thm:nfp} and Theorem~\ref{thm:mainBor}, let
$A-\lambda B$ be an $n\times n$ singular pencil of normal rank $n-k$
 with left minimal indices $n_1,\dots,n_k$ and right minimal indices $m_1,\dots,m_k$,
and let $U$, $V$, $D_A$, $D_B$, $N$, and $M$ be as in Theorem~\ref{thm:mainBor}. Since the regular part of $A-\lambda B$
then has size $r:=n-N-M-k$ and we have found $N+M+k$ new eigenvalues in Theorem~\ref{thm:mainBor}, we
have classified all eigenvalues of the perturbed pencil
\[
\widetilde A-\lambda\widetilde B:=A-\lambda B+\tau \, (UD_AV^*-\lambda \, UD_BV^*)
\]
into the following three groups:
\begin{enumerate}
\item \emph{True eigenvalues}: There are $r$ such eigenvalues that are exactly the eigenvalues
of $A-\lambda B$. The corresponding right eigenvectors $x$ and left eigenvectors $y$ satisfy
$V^*x=0$ and $U^*y=0$.
\item \emph{Prescribed eigenvalues}: There are $k$ such eigenvalues that coincide with the $k$ eigenvalues
of $D_A-\lambda D_B$. The corresponding right eigenvectors $x$ and left eigenvectors $y$ satisfy
both $V^*x\neq 0$ and $U^*y\neq 0$.
\item \emph{Random eigenvalues}: These are the remaining $N+M$ eigenvalues. They are simple and if $\mu$ is
such an eigenvalue with the corresponding right eigenvector $x$ and
left eigenvector $y$,
%, $x$ is a corresponding right eigenvector and $y$ is a corresponding left eigenvector,
then we either have $V^*x=0$ and $U^*y\neq 0$, or $V^*x\neq 0$ and $U^*y=0$.
\end{enumerate}
Thus, the eigenvalues of $A-\lambda B$ can be identified from the eigenvalues of $\widetilde A-\lambda\widetilde B$
by investigating orthogonality properties of the corresponding left and right eigenvectors.
We will use this observation in the following section for the development of an algorithm for
computing the eigenvalues of a singular square pencil.
%{\bf CM: I am in favour of taking out Remark 4.6 and 4.9, for the following reason (see also my
%comment to Remark 4.6). In the Hermitian case, I do not expect any problems, because I expect the
%values $\alpha_i$ and $\beta_i$ to be complex and conjugates of each other. In the symmetric case,
%the real eigenvalues will be double, but generically they will be $2\times 2$ Jordan blocks.
%Indeed, this is exactly what we observed in \cite{MehMW16} in the rank-1 case.
%If we have $2\times 2$ Jordan blocks, then we still have only one eigenvector (up to scaling) and
%that one should satisfy the orthogonality condition for eigenvectors associated with random eigenvalues.
%I think that it would be a wonderful new project for us to look at the structured cases next and deal
%with structure-preserving rank-completing perturbations, but for this paper, I would prefer to take
%structure out. (Concerning the structured cases, the whole theory would have to be build up again.
%In \cite{MehMW16}, we only have the rank-1 case and nothing was said on eigenvectors.)}
\end{summary}
%\medskip

\begin{remark}\rm\label{rm:symmpert2} If $A$ and $B$ are symmetric, then it seems that for our current approach
we have to %it is necessary to
use nonsymmetric rank completing perturbations.
Namely, when a symmetric perturbation of the form $\tau \, U(D_A-\lambda D_B)U^*$ is used,
there is an issue with the third group in Summary~\ref{rm:3groups}
as random eigenvalues appear either as double real eigenvalues or in complex conjugate pairs,
and in the former case the orthogonality constraints cannot be satisfied.
We leave the study of structured singular pencils for future research.
\end{remark}

%{\bf TODO: Update the paper from here. CM suggest that the method outlined in Sec. 4 is merely
%used for comparison (it just extends the ideas of Davison) and that the main focus of the paper
%is on the method in Section 5.}

\section{A perturbation method for singular generalized eigenvalue problems}
\label{sec:method}
In this section we explain how in the generic case we can extract the finite
true eigenvalues numerically even in double precision by solving only one perturbed eigenvalue problem.
The key is formed by the existent or non-existent orthogonality properties of the
left and right eigenvectors
associated with true, prescribed, and random eigenvalues, respectively.
%We can exploit these properties and obtain reliable results even in double precision.

%First, we have to determine $\text{nrank}(A,B)$. We do this
%by computing $\text{rank}(A-\zeta B)$ for a random $\zeta$.
%In the absence of more information, we select $\zeta$ random with respect to the
%normal distribution with mean 0 and variance $\sigma = \|A\|_F$.
%Although rank determination may in some instances be an ill-posed problem,
%we remark that in the context of this paper this is not expected to
%give problems, since we can simply select a new $\zeta$ value in case of
%a difficult rank determination. {\bf Todo: CM: Are we sure that this works? It may be that
%the difficulty of rank determination is more or less independent of the selected values $\zeta$,
%but a general property of the pencil. We should either test this or delete this sentence to be
%on the safe side.}

Let $A-\lambda B$ be a singular $n\times n$ pencil with normal rank $n-k$, where $k>0$.
We determine $\text{nrank}(A,B)$ by computing $\text{rank}(A-\zeta B)$ for a random $\zeta$.
As we have shown in the previous section, if
we take two random $n\times k$ matrices $U$ and $V$ with orthonormal
columns, a regular $k\times k$ diagonal pencil $D_A-\lambda D_B$, and $\tau\ne 0$, then
the perturbed pencil \eqref{pert1awithtau} is regular. The ``true'' eigenvalues of $A-\lambda B$
(theoretically) remain constant under this perturbation.
In contrast, eigenvalues that originate from the singular part
of the pencil (the ``random'' eigenvalues)
may be ``anywhere in the complex plane''. In addition, \eqref{pert1awithtau} also has $k$ ``prescribed"
eigenvalues that coincide with the eigenvalues of $D_A-\lambda D_B$.

In theory, if we compute all eigenvalues $\lambda_i$ together with
the left and right eigenvectors $x_i$ and $y_i$ for $i=1,\ldots,n$ of
\eqref{pert1awithtau}, then $\max(\|V^*x_i\|,\|U^*y_i\|)=0$ for
a true eigenvalue and $\max(\|V^*x_i\|,\|U^*y_i\|)> 0$ for
a prescribed or a random eigenvalue, so we can use this criterion
to extract the true eigenvalues.
%Once we have a set of regular eigenvalues, we can use the fact that
% $\max(\|Bx_i\|,\|B^*y_i\|)=0$ for
%an infinite eigenvalue and $\max(\|Bx_i\|,\|B^*y_i\|)>0$
%for a finite eigenvalue to further extract only the set of finite regular %eigenvalues, which usually is what we are interested in.
In the following we will discuss how the above criterion is affected by computations
in finite precision and how it does depend on $\tau$. We will
also introduce other criteria that may be used for the same purpose or
to further separate true eigenvalues into finite and infinite ones.

If $x_i$ and $y_i$ are normalized left and right eigenvectors of
the perturbed problem \eqref{pert1awithtau} for an eigenvalue $\lambda_i$, we can compute the
number
\begin{equation}
\label{eq:si}
s(\lambda_i)={y_i^*\widetilde Bx_i}=y_i^*Bx_i + \tau \, y_i^*UD_BV^*x_i.
\end{equation}
It is easy to see that $s(\lambda_i)\ne 0$ for a simple finite eigenvalue $\lambda_i$.
As explained in the following lemma, which is a straightforward
generalization of the standard result for a pencil $A-\lambda I$, see, e.g., \cite[Sec.~2.9]{WilAEP}, % or \cite[Section 7.2]{GolubVanLoan},
$1/|s(\lambda_i)|$ occurs in the expression for a standard condition number of a simple finite eigenvalue.
%\smallskip

\begin{lemma} Let $\lambda_i$ be a simple finite eigenvalue of a regular matrix pencil $\widetilde A-
\lambda \widetilde B$ and
let $x_i$ and $y_i$ be its normalized left and right eigenvectors. If we perturb the pencil into
$(\widetilde A+\theta E)-\lambda(\widetilde B+\theta F)$ for a small $\theta>0$, then $\lambda_i$ perturbs into
\begin{equation}\label{eq:obcutljivost}
\lambda_i+\theta\, {y_i^* E x_i -\lambda_i y_i^*F x_i\over s(\lambda_i)}+{\cal O}(\theta^2).
\end{equation}
\end{lemma}

If $\lambda_i$ is a simple finite true eigenvalue, then $V^*x_i=0$ and
% For a simple finite regular eigenvalue $\lambda_i$ we have $V^*x_i=0$ and
$U^*y_i=0$, which implies that
$s(\lambda_i)=y_i^*Bx_i$ does not change
with $\tau\ne 0$.
For a regular infinite eigenvalue we have $y_i^* B=0$, $Bx_i=0$, $V^*x_i=0$, and
$U^*y_i=0$, therefore $s(\infty)=0$, again independent of
$\tau\ne 0$.
On the other hand, we can show that values $s(\lambda)$ of prescribed
and random eigenvalues depend on $\tau$ and go to $0$ as
$\tau$ goes to 0. For this, we need the following lemma.

\begin{lemma}\label{lem:znano}
Let $A-\lambda B$ be a singular pencil and let $\alpha$ be different
from all eigenvalues of $A-\lambda B$, i.e.,
$\text{rank}(A-\alpha B)={\rm nrank}(A,B)$. If
$y^*(A-\alpha B)=0$ and $(A-\alpha B)x=0$ then
$y^*Ax=y^*Bx=0$.
\end{lemma}

\begin{proof}
We know from the structure of the left and right singular blocks that
$x\in {\cal N}_r(\alpha)$ can be written as
$x=q_1+\alpha q_2 +\cdots + \alpha^p q_{p+1}$,
where the vectors $q_1,\ldots,q_{p+1}$ form the chain
\[ Aq_1 =0,\ Aq_2 = Bq_1,\ \ldots,\ Aq_{p+1}=Bq_p,\ Bq_{p+1}=0\]
for certain $p\ge 0$.
Similarly,
$y\in {\cal N}_l(\alpha)$ can be written as
$y=w_1+\overline \alpha w_2 +\cdots + \overline\alpha^r w_{r+1}$,
where the vectors $w_1,\ldots,w_{r+1}$ form the chain
\[
A^*w_1  =0,\  A^*w_{2} = B^*w_1,\ \ldots,\ A^*w_{r+1} = B^*w_r,\ B^*w_{r+1}=0
\]
for certain $r\ge 0$.
To show $y^*Bx=0$ it is enough to show
that $w_i^*Bq_j=0$ for all $i=1,\ldots,r+1$ and $j=1,\ldots,p+1$.
%For $j=p+1$ this follows from $Bq_{p+1}=0$, so we can assume that $j\le p$. It follows that $w_i^*Bq_j=w_i^*Aq_{j+1}$. If $i=1$, then $w_1^*A=0$, if not,
%we can continue to $w_i^*Bq_j=w_i^*Aq_{j+1}=w_{i-1}^*Bq_{j+1}.$ As we continue
%in this manner, we eventually reach either $w_1^*A=0$ or $Bq_{p+1}=0$. It follows that $y^*Bx=0$ and from $y^*Ax=\alpha y^*Bx$ we get
%$y^*Ax=0$ as well.
For $i=1$ or $j=p+1$ this follows from $w_1^*A=0$ and $Bq_{p+1}=0$, so we can assume that $i\ge 2$ and $j\le p$. It follows that $w_i^*Bq_j=w_i^*Aq_{j+1}=w_{i-1}^*Bq_{j+1}$. As we continue
in this manner, we eventually reach either $w_1^*A=0$ or $Bq_{p+1}=0$. It follows that $y^*Bx=0$ and from $y^*Ax=\alpha y^*Bx$ we get
$y^*Ax=0$ as well.
\end{proof}
%\smallskip

\begin{lemma}\label{lem:spol} Let $\lambda_i$ be a prescribed or random eigenvalue of
\eqref{pert1awithtau} under the assumptions of Theorem~{\rm\ref{thm:mainBor}}, where we assume in addition
that all prescribed eigenvalues are algebraically simple.
Then there exists a positive constant $c_i$ such that
$|s(\lambda_i)|={c_i |\tau|}$.
\end{lemma}

\begin{proof}
First, let $\lambda_i$ be a prescribed eigenvalue. Then by the proof of Theorem~\ref{thm:mainBor} the
corresponding left and right eigenvectors satisfy $y_i^*(A-\lambda_i B)=0$ and $(A-\lambda_i B)x_i=0$
which by Lemma~\ref{lem:znano} implies $y_i^*Bx_i=0$. But then we have
$|s(\lambda_i)|=c_i|\tau|$ with $c_i=|y_i^*UD_BV^*x_i|$ and $c_i$ must be nonzero, because $\lambda_i$ is
a simple eigenvalue of $\widetilde A-\lambda\widetilde B$ for $\tau\neq 0$.

Next, let $\lambda_i$ be a random eigenvalue, such that $V^*x_i=0$ and $U^*y_i\ne 0$.
We know (see the proof of Theorem~\ref{thm:mainBor}) that $y_i$ is a linear function of $\tau$ as
$y_i(\tau)=w_i+\tau z_i$, where
$w_i^*(A-\lambda_i B)=0$. Since $(A-\lambda_i B)x_i=0$, it follows
from Lemma \ref{lem:znano} that $w_i^*Bx_i=0$ and
$y_i^*\widetilde Bx_i=y_i^*Bx_i=\overline \tau z_i^*Bx_i$.
The case $V^*x_i\ne 0$ and $U^*y_i= 0$ can be shown analogously.
\end{proof}
%\smallskip

So, if we take a $\tau$ of small absolute value and if all finite true eigenvalues are simple and none of them is too ill-conditioned,
then we can separate the finite true eigenvalues from the remaining ones
using
%with the help of
the values $s(\lambda)$.

Let $\varepsilon$ be the machine precision and let the
matrices $A$ and $B$ be scaled in such way that $\|A\|=\|B\|=1$.
If all finite true eigenvalues are simple and not too ill-conditioned,
then we expect to observe the situation in Table~\ref{tab1},
where $c>0$ is a constant, independent of $\tau$, and possibly different for each eigenvalue and each entry in the table.

\begin{center}
\begin{table}[h!]
\caption{Characteristics of the eigenvalues of the perturbed pencil as in~\eqref{pert1awithtau}.}
\vspace{-2mm}
\hspace*{1.2cm}\begin{tabular}{l|l|l|l} \hline
Eigenvalue $\lambda$ & $|s(\lambda)|$ & $\|V^*x\|$ & $\|U^*y\|$ \\
\hline \rule{0pt}{2.3ex}%
Finite true eigenvalue of $A-\lambda B$& $c$ & $< c \,\varepsilon/|\tau|$ & $< c \,\varepsilon/|\tau|$ \\
Infinite true eigenvalue of $A-\lambda B$ & $< c \,\varepsilon$ & $< c \,\varepsilon/|\tau|$ & $< c \,\varepsilon/|\tau|$
\\
Prescribed eigenvalue of $D_A-\lambda D_B$ & $c \, |\tau|$ & $ c$ & $ c$\\
Random eigenvalue from an $L_p$ block & $ c \, |\tau|$ & $< c \,\varepsilon/|\tau|$ & $ c$\\
Random eigenvalue from an $L_p^T$ block & $ c \, |\tau|$ & $ c$ & $< c \,\varepsilon/|\tau|$\\
\hline
\end{tabular}
%\smallskip
\label{tab1}
\end{table}
\end{center}

%In the above table $\approx$ denotes that the value is nonzero and approximately equal to the given right-hand expression,
%while $\lessapprox$ denotes that the right-hand expression is a bound.
We now explain the values in Table~\ref{tab1}.
We will start with column $|s(\lambda)|$ and a finite true eigenvalue, where we assume that all
finite true eigenvalues are simple and well-conditioned.
It follows that $\lambda$ is a simple eigenvalue of $\widetilde A-\lambda \widetilde B$, therefore $y^*\widetilde Bx\ne 0$ and,
since this value is independent of $\tau$ and $\varepsilon$, we have $|s(\lambda)|=c$.
For an infinite eigenvalue we should have $y^*\widetilde Bx=0$ in
 exact computation, instead, in finite precision, we get $|y^*\widetilde Bx|< c \,\varepsilon$.
Finally, in the generic case, if $\lambda_i$ is a prescribed or random eigenvalue then Lemma~\ref{lem:spol} yields that
$|s(\lambda_i)|=c_i \, |\tau|$ for a positive constant $c_i$.

Finally, % we turn to
the values in the columns $\|V^*x\|$ and $\|U^*x\|$ that are
marked by $<c \,\varepsilon/|\tau|$ should be zero in exact arithmetic.
In finite precision however, due to the supposed backward stability of the applied eigenproblem solver, the
computed eigenvalues and eigenvectors of $\widetilde A-\lambda \widetilde B$
are exact eigenpairs of a perturbed pencil
$\widetilde A + E-\lambda \, (\widetilde B+F)$, where $\|E\|\le \widetilde c_1\,\|\widetilde A\|\,\varepsilon$
and $\|F\|\le \widetilde c_2\,\|\widetilde B\|\,\varepsilon$.
If we assume that all finite eigenvalues of $\widetilde A-\lambda \widetilde B$ are simple,
then we have the following result on the first-order eigenvector perturbations.
The proof is omitted since it is a straightforward
generalization of the result for the pencil $A-\lambda I$ from \cite[Sec.~2.10]{WilAEP}.
%\smallskip

\begin{lemma}\label{lem:pertvecab}
Let all finite eigenvalues $\lambda_i$ of $\widetilde A-\lambda \widetilde B$ be simple and let
$x_i$ and $y_i$ be corresponding left and right normalized eigenvectors. If
the pencil is perturbed into $\widetilde A + \theta E-\lambda \, (\widetilde B+\theta F)$, then
the eigenvector $x_i$ perturbs into %is perturbed into
\[
\widetilde x_i=x_i+\theta \sum_{k=1, k\ne i}^n {y_k^*(E-\lambda_i F)x_i\over (\lambda_i-\lambda_k)\,s(\lambda_k)}\,x_k+
{\cal O}(\theta^2).
\]
\end{lemma}

%\begin{proof}
%We can write \begin{equation}
%(\widetilde A+\theta E)\left(x_i+\sum_{k=1,k\ne i}^n \gamma_k x_k\right)=
%(\lambda_i + \delta \lambda_i)(\widetilde B+\theta E)\left(x_i+\sum_{k=1,k\ne i}^n \gamma_k x_k\right).
%\label{eq:vecpert}
%\end{equation}
%We know that $y_j^*\widetilde Ax_i=y_j^*\widetilde Bx_i=0$ for $i\ne j$, if we multiply
%\eqref{eq:vecpert} by $y_j^*$, where $i\ne j$, we get
%$$\gamma_j=\theta {y_j^*(E-\lambda_i F)x_i\over (\lambda_i-\lambda_j)y_j^*Bx_j}+
%{\cal O}(\theta^2).$$
%\end{proof}

Let $\lambda_i$ be a finite true eigenvalue of $A-\lambda B$. Then $\lambda_i$ is also
an eigenvalue of $\widetilde A-\lambda\widetilde B$ and $V^*x_i=0$, where $x_i$ is an exact
normalized right eigenvector. In finite precision, $x_i$ becomes perturbed in the directions of other
eigenvectors and by Lemma~\ref{lem:pertvecab} a contribution in the direction of another
eigenvector depends on the condition number of the corresponding eigenvalue. The only contributions
that affect the value of $\|V^*\widetilde x_i\|$ are those related to prescribed eigenvalues or
random eigenvalues from left singular blocks, as right eigenvectors of other eigenvalues are orthogonal
to $V$. As condition numbers of these eigenvalues are equal to $1/(c\,|\tau|)$ and
$\|V^*x_j\|=c$ for the corresponding right eigenvectors, it follows from
Lemma \ref{lem:pertvecab} and the backward stability of the computed eigenpairs that
$\|V^*\widetilde x_i\|< c \, \varepsilon/|\tau|$.

%After finishing the explanations on the values in Table~\ref{tab1},
Next, we discuss appropriate % turn to the question
choices for the value $\tau$.
If %we use a $\tau$ such that
$|\tau|$ is close to $\varepsilon$, then the prescribed and random eigenvalues are very
ill-conditioned, and perturbations of eigenvectors may move the values of $\|V^*x_i\|$ and $\|U^*y_i\|$
far away from zero %none of these values may be close to zero,
when they should be close to zero. Therefore, if $|\tau|$ is too small, we may not be able to use
the values of $\|V^*x_i\|$ and $\|U^*y_i\|$ to extract the true eigenvalues. Still, if all finite true eigenvalues
are simple, then we may use the values $|s(\lambda_i)|$ to extract the finite true eigenvalues.

On the other hand, if $|\tau|$ is large,
then all eigenvalues, except the infinite ones, are expected to be well-conditioned which means
that the eigenvectors will not change much and the computed left and right eigenvectors will be
orthogonal to $V$ or $U$ in finite precision, when they should be. Therefore, for large $|\tau|$, we can first use
$\max(\|V^*x_i\|,\|U^*y_i\|)$ to extract the true eigenvalues and then use $|s(\lambda_i)|$ to distinguish
the finite true eigenvalues from the infinite one. In practice, we see this as a better option,
because it does not depend on finite true eigenvalues being simple.
However, we should not choose $|\tau|$ too large as this may decrease the precision of
the computed finite true eigenvalues. Since the computed eigenvalues
are, due to assumed backward
stability, exact eigenvalues of a slightly perturbed pencil $\widetilde A-\lambda\widetilde B$, it is safe to use $|\tau|$ up to
$\|\widetilde A\|\approx \|A\|$ and $\|\widetilde B\|\approx \|B\|$.
Also, since from our analysis it follows that only the absolute value of $\tau$
seems to matter, we suggest to choose $\tau$ real and positive.

%a safe maximum choice would for instance be
%\[\tau={1\over 10}\min\left({\|A\|\over \|D_A\|},{\|B\|\over %\|D_B\|}\right).\]

%A rule of thumb value that enables one to use both criteria is
%$\tau=\varepsilon^{1/2}$. We can also combine both criteria by defining
%\[c(\lambda):=\kappa(\lambda)\max(\|V^*x\|,\|U^*y\|)\]
%and use this value to separate finite regular eigenvalues from the remaining ones.
%
%In double precision, where $\varepsilon=10^{-16}$, if we take $\tau=10^{-8}$,
%we get something like the following in a nice situation.
%\begin{center}
%\begin{tabular}{l|l|l|l|l}
%eigenvalue $\lambda$ & $\kappa(\lambda)$ & $\|V^*x\|$ & $\|U^*y\|$ & $c(\lambda)$ \\
%\hline
%finite regular eigenvalue of $A-\lambda B$ & ${\cal O}(1)$ & ${\cal O}(10^{-8})$ & ${\cal O}(10^{-8})$ & ${\cal O}(10^{-8})$\\
%infinite regular eigenvalue of $A-\lambda B$ & ${\cal O}(10^{16})$ & ${\cal O}(10^{-8})$ & ${\cal O}(10^{-8})$ & ${\cal O}(10^{8})$
%\\
%fixed eigenvalue of $D_A-\lambda D_B$ & ${\cal O}(10^{8})$ & ${\cal O}(1)$ & ${\cal O}(1)$ &
%${\cal O}(10^{8})$\\
%fake eigenvalue from an $L_p$ block & ${\cal O}(10^{8})$ & ${\cal O}(10^{-8})$ & ${\cal O}(1)$ & ${\cal O}(10^{8})$\\
%fake eigenvalue from an $L_p^T$ block & ${\cal O}(10^{8})$ & ${\cal O}(1)$ & ${\cal O}(10^{-8})$ & ${\cal O}(10^{8})$\\
%\hline
%\end{tabular}
%\end{center}
%\medskip

Based on the above discussion, we summarize our method in Algorithm~1. Note that we scale the matrices in such way that $\|A\|_1=\|B\|_1=1$, mainly for convenience, to determine an appropriate default value
for $\tau$.

\noindent\vrule height 0pt depth 0.5pt width \textwidth \\
{\bf Algorithm~1: Computing finite eigenvalues of a singular pencil $(A,B)$}
by a rank-completing perturbation. \\[-3mm]
\vrule height 0pt depth 0.3pt width \textwidth \\
{\bf Input:} $A$ and $B$, perturbation constant $\tau$ (default $10^{-2}$),
thresholds $\delta_1$ (default $\varepsilon^{1/2}$)
and $\delta_2$ (default $10^{2}\,\varepsilon$).\\
{\bf Output:} Eigenvalues of the finite regular part. \\
\begin{tabular}{ll}
{\footnotesize 1:} & Scale $A=(1/\alpha)A$ and $B=(1/\beta )B$, where $\alpha=\|A\|_1$ and $\beta=\|B\|_1$. \\
{\footnotesize 2:} & Compute $\text{nrank}(A,B)$:
$k = \text{rank}(A-\zeta B)$ for random $\zeta$. \\
{\footnotesize 3:} & Select random $n\times k$ matrices $U$ and $V$ with orthonormal columns. \\
{\footnotesize 4:} & Select diagonal $k\times k$ matrices $D_A$ and $D_B$ such that the \\
& eigenvalues of $(D_A, D_B)$ are (likely) different from those of $(A,B)$ \\
& (default: choose diagonal elements of $D_A$ and $D_B$ uniformly random  \\ & from the interval $[1,2]$).\\
{\footnotesize 5:} & Compute the eigenvalues $\lambda_i$, $i=1,\ldots,n$, and right and left\\ & eigenvectors $x_i$ and $y_i$ of $(\widetilde A,\widetilde B)=(A+\tau \, UD_AV^*,\, B+\tau \, UD_BV^*).$ \\
{\footnotesize 6:} & Compute
$s_i=y_i^*\widetilde Bx_i$ for $i=1,\ldots,n$.\\
{\footnotesize 7:} & Compute $\zeta_i=\max(\|V^*x_i\|, \, \|U^*y_i\|)$
for $i=1,\ldots,n$.\\
{\footnotesize 8:} & Return all eigenvalues $(\alpha/\beta)\lambda_i$, % for
$i=1,\ldots,n$, where $\zeta_i<\delta_1$ and $|s_i|>\delta_2$.
\end{tabular} \\
\vrule height 0pt depth 0.5pt width \textwidth
\medskip

As we will show by experiments in the next section,
the above approach seems to work very well in double precision for small or moderate singular pencils.
Of course, if some of the eigenvalues are very ill-conditioned (for instance when some of the
eigenvalues are multiple), then the method may fail in extracting some of the finite true
eigenvalues. However, its advantage over staircase-based methods may be the following observation:
if we make a wrong rank decision in a staircase algorithm, then the method usually fails completely and returns no eigenvalues at all;
see Example~\ref{ex:num3} in the next section.
In contrast, the method proposed here is able to detect, if not all, then at least the well-conditioned finite
eigenvalues of the pencil under consideration.

%In such case there are some other things that we can do.
%We noticed that left and right eigenvectors for a regular eigenvalue or eigenvalue 1
%do not change with $\tau$, while one of the eigenvectors (left or right) of a fake eigenvalue changes with $\tau$.
%So, after we compute eigenvalues and eigenvectors for an initial $\tau$,
%we can compute the left and right eigenvectors for a different $\widetilde \tau$
%for all eigenvalues where the above criteria does not give a clear answer.
%As we already know the eigenvalues (they are independent of $\tau$) we only have to compute the eigenvectors,
%which might be cheaper than using the eigensolver again.

\section{Numerical examples}\label{sec:num} In this section we demonstrate the method with several numerical examples computed
 in Matlab 2015b.
 All numerical examples and implementations of the algorithms are available in \cite{MultiParEig}.

 %In order to confirm the behaviour or the computed eigenvalues and %eigenvectors from the previous section, some of the results
%were computed in higher precision using Advanpix Multiprecision Computing %Toolbox \cite{Advanpix}.

\begin{example}\rm\label{ex:num1} We take a $7\times 7$ pencil $A-\lambda B$, where
\[A=\smath{\left[\begin{array}{rrrrrrr}
  -1 & -1 & -1 & -1 & -1 & -1 & -1 \\
  1 & 0 & 0 & 0 & 0 & 0 & 0 \\
  1 & 2 & 1 & 1 & 1 & 1 & 1 \\
  1 & 2 & 3 & 3 & 3 & 3 & 3 \\
  1 & 2 & 3 & 2 & 2 & 2 & 2 \\
  1 & 2 & 3 & 4 & 3 & 3 & 3 \\
  1 & 2 & 3 & 4 & 5 & 5 & 4\end{array}\right]},\quad
B=\smath{\left[\begin{array}{rrrrrrr}
  -2 & -2 & -2 & -2 & -2 & -2 & -2 \\
  2 & -1 & -1 & -1 & -1 & -1 & -1 \\
  2 & 5 & 5 & 5 & 5 & 5 & 5 \\
  2 & 5 & 5 & 4 & 4 & 4 & 4 \\
  2 & 5 & 5 & 6 & 5 & 5 & 5 \\
  2 & 5 & 5 & 6 & 7 & 7 & 7 \\
  2 & 5 & 5 & 6 & 7 & 6 & 6\end{array}\right]}.
\]
The matrices are built in such way that
the KCF of the pencil contains blocks of all four possible types,
${\rm nrank}(A,B)=6$ and the pencil is singular. Its
KCF has blocks $J_1(1/2)$, $J_1(1/3)$, $N_1$, $L_1$, and $L_2^T$. If we apply Algorithm~1, we get
the values in the following table. Note that values $\lambda_k$ in the first column are values from Line 8 of Algorithm 1, which are scaled back to match the eigenvalues of the original matrix pencil $A-\lambda B$ whose matrices are scaled in Line 1 since they do not satisfy
$\|A\|_1=\|B\|_1=1$.

\vspace{1mm}
\begin{center}
{\footnotesize \begin{tabular}{r|clll} \hline
$k$ & $\lambda_k$ & $\quad \ |s_k|$ & $\ \|V^*x_k\|$ & $\ \|U^*y_k\|$ \\
\hline \rule{0pt}{2.3ex}%
1 & 0.333333 & $1.5\cdot 10^{-2}$ & $1.3\cdot 10^{-15}$ & $1.3\cdot 10^{-14}$ \\
2 & 0.500000 & $9.5\cdot 10^{-4}$ & $1.3\cdot 10^{-14}$ & $1.9\cdot 10^{-14}$ \\
3 & $\infty$ & $3.8\cdot 10^{-19}$ & $2.8\cdot 10^{-15}$ & $1.3\cdot 10^{-14}$  \\
4 & $-0.244794 + 0.421723i$ & $7.8\cdot 10^{-3}$ & $5.8\cdot 10^{-2}$ & $5.6\cdot 10^{-15}$ \\
5 & $-0.244794 - 0.421723i$ & $7.8\cdot 10^{-3}$ & $5.8\cdot 10^{-2}$ & $5.6\cdot 10^{-15}$  \\
6 & 0.383682 & $2.1\cdot 10^{-4}$ & $2.6\cdot 10^{-2}$ & $4.2\cdot 10^{-1}$ \\
7 & 0.478292 & $2.6\cdot 10^{-4}$ & $9.2\cdot 10^{-15}$ & $5.2\cdot 10^{-1}$ \\
\hline
\end{tabular}}
\end{center}
\smallskip

The values in the table follow the pattern from the previous section and it is easy to detect that $\lambda_1$ and $\lambda_2$
are finite true eigenvalues, $\lambda_3$ is a true infinite eigenvalue, $\lambda_4,\lambda_5$, and $\lambda_7$ are random eigenvalues, and $\lambda_6$ is the prescribed eigenvalue.
\end{example}
%\smallskip

\begin{example}\rm\label{ex:dk3} We take
example C3 from \cite{DKa88} that comes from control theory and belongs to a set of examples C1, C2, and C3,
where each has successively more ill-conditioned eigenvalues. The pencil has the form
\[A-\lambda B =
\smath{\left[\begin{array}{rrrrr}
  1 & -2 & 100 & 0 & 0 \\
  1 & 0 & -1 & 0 & 0 \\
  0 & 0 & 0 & 1 & -75  \\
  0 & 0 & 0 & 0 & 2\end{array}\right]}
-\lambda \, \smath{\left[\begin{array}{rrrrr}
  0 & 1 & 0 & 0 & 0 \\
  0 & 0 & 1 & 0 & 0 \\
  0 & 0 & 0 & 1 & 0  \\
  0 & 0 & 0 & 0 & 1\end{array}\right]}.
\]
Its KCF contains blocks $L_2$, $J_1(1)$, and $J_1(2)$.
As the pencil is rectangular, we add a zero line to make it square.
This adds an $L_0^T$ block to the KCF.
Algorithm~1 returns the following table for $A-\lambda B$, from which the finite true
eigenvalues $1$ and $2$ can be extracted.

\vspace{1mm}
\begin{center}
{\footnotesize \begin{tabular}{r|clll} \hline
$k$ & $\lambda_k$ & $\quad \ |s_k|$ & $\ \|V^*x_k\|$ & $\ \|U^*y_k\|$ \\
\hline \rule{0pt}{2.3ex}%
1 & 1.000000 & $1.2\cdot 10^{-2}$ & $1.7\cdot 10^{-15}$ & $1.9\cdot 10^{-15}$ \\
2 & 2.000000 & $1.2\cdot 10^{-2}$ & $2.3\cdot 10^{-15}$ & $1.7\cdot 10^{-15}$ \\
3 & $-0.693767 + 1.563033i$ & $2.1\cdot 10^{-2}$ & $2.3\cdot 10^{-15}$ & $5.0\cdot 10^{-1}$  \\
4 & $-0.693767 - 1.563033i$ & $2.1\cdot 10^{-2}$ & $2.3\cdot 10^{-15}$ & $5.0\cdot 10^{-1}$  \\
5 & 78.673901 & $2.8\cdot 10^{-3}$ & $3.3\cdot 10^{-1}$ & $6.4\cdot 10^{-1}$ \\
\hline
\end{tabular}}
\end{center}
\smallskip

As in \cite{DKa88} we add some noise and perturb initial
$A-\lambda B$ into $\widehat A-\lambda \widehat B$
by adding $10^{-6}\,{\tt rand}(4,5)$ to $A$ and $B$. True
eigenvalues of $\widehat A - \lambda \widehat B$
 can still be extracted by Algorithm~1 if we adjust the
parameter $\delta_1$. The values we get are in the following table:

\vspace{1mm}
\begin{center}
{\footnotesize \begin{tabular}{r|rlll} \hline
$k$ & $\lambda_k \quad \ $ & $\quad \ |s_k|$ & $\ \|V^*x_k\|$ & $\ \|U^*y_k\|$ \\
\hline \rule{0pt}{2.3ex}%
1 & 0.999990 & $7.6\cdot 10^{-3}$ & $2.6\cdot 10^{-15}$ & $5.2\cdot 10^{-7}$ \\
2 & 2.000058 & $7.6\cdot 10^{-3}$ & $2.8\cdot 10^{-15}$ & $5.6\cdot 10^{-7}$ \\
3 & 101.850555 & $8.2\cdot 10^{-4}$ & $1.3\cdot 10^{-1}$ & $3.7\cdot 10^{-1}$  \\
4 & $-14.308508$ & $1.9\cdot 10^{-2}$ & $5.8\cdot 10^{-16}$ & $4.5\cdot 10^{-1}$  \\
5 & 15.734162 & $9.3\cdot 10^{-3}$ & $1.0\cdot 10^{-17}$ & $4.7\cdot 10^{-1}$ \\
\hline
\end{tabular}}
\end{center}
%\smallskip
\end{example}

\medskip % for old style
\begin{example}\rm\label{ex:em} This is an example from \cite[Sec.~5]{EM},
where the staircase algorithm fails to find a regular subspace of proper size under a small random perturbation. We take
\[A_1-\lambda B_1=
\smath{\left[\begin{array}{cccc} 0 & 0 & 1 & 0\cr 0 & 0 & 0 & 1\cr 0 & 0 & 0 & 0\end{array}\right]} -\lambda \,
\smath{\left[\begin{array}{cccc} \delta & 0 & 0 & 0\cr 0 & \delta & 0 & 0\cr 0 & 0 & 1 & 0\end{array}\right]},
\]
where $\delta=1.5\cdot 10^{-8}$. The KCF structure of the pencil is  $J_2(0)$ and $L_1$ which means that $0$ is a double eigenvalue.
It is reported in \cite{EM} that if we add a random perturbation of
size $10^{-14}$ to the pencil, then Guptri reports the regular part $J_1(0)$ and we
have been able to confirm this using a Matlab implementation
of Guptri in \cite{MCS}. If we enlarge the perturbation to $10^{-11}$, Guptri returns
no regular part at all, while Algorithm~1 returns two finite true eigenvalues $\lambda_1$ and $\lambda_2$ from the following table.

\vspace{1mm}
\begin{center}
{\footnotesize \begin{tabular}{r|llll} \hline
$k$ & $\qquad \ \lambda_k$ & $\quad \ |s_k|$ & $\ \|V^*x_k\|$ & $\ \|U^*y_k\|$ \\
\hline \rule{0pt}{2.3ex}%
1 & $-1.4306543\cdot 10^{-3}$ & $1.6\cdot 10^{-11}$ & $5.5\cdot 10^{-17}$ & $6.7\cdot 10^{-10}$ \\
2 & $\phantom-9.9599790\cdot 10^{-4}$ & $1.6\cdot 10^{-11}$ & $0.0$ & $6.7\cdot 10^{-10}$ \\
3 & $-2.2641370\cdot 10^{7}$ & $5.2\cdot 10^{-9}$ & $2.9\cdot 10^{-18}$ & $2.6\cdot 10^{-6}$  \\
4 & $\phantom-1.1878888\cdot 10^0$ & $1.6\cdot 10^{-3}$ & $2.0\cdot 10^{-1}$ & $7.8\cdot 10^{-1}$  \\
\hline
\end{tabular}}
\end{center}
%\smallskip
\end{example}

\medskip % for old style
\begin{example}\rm\label{ex:num3} We take the singular pencil $\Delta_1-\lambda \Delta_0$
of size $300\times 300$ from
\cite[Ex.~18]{MPl14}. This example is related to two random matrices
$A$ and $B$ of size $10\times 10$ in a way that the true eigenvalues of $\Delta_1-\lambda \Delta_0$
are exactly the values $\lambda$ such that $A+\lambda B$ has a multiple eigenvalue (see Section~\ref{sec:double}).
We know from the
properties of the problem that there are 90 such values $\lambda$ and that the KCF of $\Delta_1-\lambda \Delta_0$
contains $100$ $N_1$ and 10 left and 10 right singular blocks. The conjecture from \cite{MPl14}
is that the singular blocks are 5 $L_4^T$, 5 $L_5^T$, 5 $L_5$, and 5 $L_6$ blocks.

This example is also available as
\verb|demo_double_eig_mp| in toolbox MultiParEig \cite{MultiParEig}. The staircase algorithm in
MultiParEig fails to extract the finite regular part of size 90 in double precision, but manages
to extract all 90 finite true eigenvalues using quadruple precision and the Multiprecision Computing Toolbox \cite{Advanpix}.
If we apply Algorithm~1 to $\Delta_1-\lambda \Delta_0$ in double precision, we get the following values:

\vspace{1mm}
\begin{center}
{\footnotesize \begin{tabular}{c|clll} \hline
$k$ & $\lambda_k$ & $\quad \ |s_k|$ & $\ \|V^*x_k\|$ & $\ \|U^*y_k\|$ \\
\hline \rule{0pt}{2.3ex}%
1 & $0.508999 + 2.016378i$ & $3.0\cdot 10^{-3}$ & $2.3\cdot 10^{-14}$ & $1.6\cdot 10^{-14}$ \\
$\vdots$ & $\vdots$ & $\quad \ \vdots$ & $\quad \ \vdots$ & $\quad \ \vdots$ \\
89 & $4.266290-0.925962i$ & $1.4\cdot 10^{-6}$ & $2.6\cdot 10^{-13}$ & $7.2\cdot 10^{-14}$ \\
90 & $-0.628208$ & $3.2\cdot 10^{-7}$ & $2.8\cdot 10^{-14}$ & $1.3\cdot 10^{-11}$ \\
91 & $\infty$ & $1.1\cdot 10^{-17}$ & $7.1\cdot 10^{-15}$ & $7.1\cdot 10^{-15}$ \\
$\vdots$ & $\vdots$ & $\quad \ \vdots$ & $\quad \ \vdots$ & $\quad \ \vdots$ \\
190 & $\infty$ & $2.8\cdot 10^{-21}$ & $5.9\cdot 10^{-15}$ & $7.9\cdot 10^{-15}$ \\
191 & $-6.276934$ & $3.2\cdot 10^{-7}$ & $2.7\cdot 10^{-14}$ & $4.5\cdot 10^{-5}$ \\
$\vdots$ & $\vdots$ & $\quad \ \vdots$ & $\quad \ \vdots$ & $\quad \ \vdots$ \\
300 & 7.125982 & $2.3\cdot 10^{-5}$ & $1.7\cdot 10^{-1}$ & $1.0\cdot 10^{-2}$ \\
\hline
\end{tabular}}
\end{center}
\vspace{1mm}
From the columns $\|V^*x_k\|$ and $\|U^*y_k\|$ we get $\max_{k=1,\ldots,190}(\max(\|V^*x_k\|,\|U^*y_k\|))=1.3\cdot 10^{-11}$
and $\min_{k=191,\ldots,300} \max(\|V^*x_k\|,\|U^*y_k\|)=4.5\cdot 10^{-5}$,
which shows a clear gap which separates true eigenvalues from the prescribed and random ones.
Next, in the set of true eigenvalues
there is also a clear gap between $s_{90}$ and $s_{91}$ which separates finite true eigenvalues from infinite ones, since
$\min_{k=1,\ldots,90}|s_k|=3.2\cdot 10^{-7}$ and
$\max_{k=91,\ldots,190}|s_k|=1.1\cdot 10^{-17}$.
\end{example}

\section{The singular two-parameter eigenvalue problem}\label{sec:2EP}
We now expand on Section~\ref{sub:sing2EP}.
In a two-parameter eigenvalue problem (2EP) \cite{Atk72} we have the equations
\begin{align}
(A_1+\lambda B_1 +\mu C_1) \, \bx_1 &= \zero,\nonumber\\[-3mm]
  \label{problem} \\[-3mm]
(A_2+\lambda B_2 +\mu C_2) \, \bx_2 &= \zero, \nonumber
\end{align}
where $A_1$, $B_1$, and $C_1$ are of size $n_1 \times n_1$,
and $A_2$, $B_2$, and $C_2$ are of size $n_2 \times n_2$.
Sought are scalars $\lambda, \mu$ and nonzero vectors $\bx_1$ and $\bx_2$
such that \eqref{problem} is satisfied. We say
that $(\lambda,\mu)$ is an eigenvalue of the 2EP and the tensor product $\bx_1 \otimes \bx_2$
is the corresponding eigenvector. Define the operator determinants
\begin{align}
  \Delta_0&=B_1\otimes C_2-C_1\otimes B_2,\nonumber\\
  \Delta_1&=C_1\otimes A_2-A_1\otimes C_2,\label{D012}\\
  \Delta_2&=A_1\otimes B_2-B_1\otimes A_2\nonumber.
\end{align}
Then problem \eqref{problem} is related to a coupled pair of GEPs
\begin{align}
  \Delta_1 \, \bz &=\lambda \, \Delta_0 \, \bz,\nonumber\\[-3mm]
  \label{Dgep}\\[-3mm]
  \Delta_2 \, \bz &=\mu \, \Delta_0 \, \bz \nonumber
\end{align}
for a decomposable tensor $\bz = \bx_1 \otimes \bx_2$.
If $\Delta_0$ is nonsingular, then Atkinson~\cite{Atk72} shows that
the solutions of \eqref{problem} and \eqref{Dgep} agree and the
matrices $\Delta_0^{-1}\Delta_1$ and $\Delta_0^{-1}\Delta_2$ commute.
In the nonsingular case the 2EP \eqref{problem} has
$n_1n_2$ eigenvalues and it can be solved with a variant
of the QZ algorithm on \eqref{Dgep}; see \cite{HKP05}.

It turns out that for many problems occurring in practice
both pencils $(\Delta_1, \Delta_0)$ and $(\Delta_2, \Delta_0)$ are singular and we
have a singular 2EP \cite{MPl09}.
Applications include delay-differential equations \cite{JHo09},
quadratic two-parameter eigenvalue problems \cite{MPl10,HMP12},
model updating \cite{Cot01},
and roots of systems of bivariate polynomials \cite{PHo16,BDD17}.

The eigenvalues of a singular 2EP \eqref{problem} are the finite regular
eigenvalues of \eqref{Dgep}; see Section~\ref{sub:sing2EP}.
There exists a staircase type algorithm that works on both
singular pencils \eqref{Dgep} simultaneously and extracts finite regular eigenvalues;
see \cite{MPl10} and an implementation in \cite{MultiParEig}. However, as illustrated in
Examples~\ref{ex:em} and \ref{ex:num3}, a staircase algorithm may fail. In this section
we propose an alternative method that may be applied to a singular 2EP,
which in some cases finds finite regular eigenvalues when the staircase algorithm fails,
while in some other cases the situation is exactly the opposite.

We can apply Algorithm~1 to $\Delta_1 z = \lambda \Delta_0 z$, one of the two singular pencils in \eqref{Dgep},
to compute the $\lambda_i$ components of eigenvalues $(\lambda_i,\mu_i)$. This is, however, only half of
the required information and for each $\lambda_i$ we have to find the corresponding $\mu_i$.
Subsequently, we insert $\lambda=\lambda_i$ into \eqref{problem} and search for
common eigenvalues $\mu$ of a pair of pencils
$(A_1-\lambda_i B_1)- \mu C_1$ and $(A_2-\lambda_i B_2)- \mu C_2$
that may be singular as well.
We detect the common eigenvalues by comparing the sets of computed eigenvalues for the first and the second pencil,
for which we use Algorithm~1 again. The overall method is given in Algorithm~2.

\noindent\vrule height 0pt depth 0.5pt width \textwidth \\
{\bf Algorithm~2: Computing finite regular eigenvalues of a singular 2EP} \\[-3mm]
\vrule height 0pt depth 0.3pt width \textwidth \\
{\bf Input:} Matrices $A_1,B_1,C_1,A_2,B_2,C_2$ from \eqref{problem}
which provide $\Delta_1$ and $\Delta_0$ from \eqref{D012};
threshold $\delta$ (default $\delta=\varepsilon^{1/2}$), and parameters for Algorithm~1.\\
{\bf Output:} Finite regular eigenvalues of \eqref{problem}. \\
\begin{tabular}{ll}
%{\footnotesize 1:} & Compute $\Delta_1=C_1\otimes A_2-A_1\otimes C_2$ and
%$\Delta_0 = B_1\otimes C_2-C_1\otimes B_2$.\\
{\footnotesize 1:} & Compute finite eigenvalues $\lambda_1,\ldots,\lambda_r$ of
$\Delta_1 - \lambda \Delta_0$ using Algo.~1.\\
{\footnotesize 2:} & {\bf for} $j=1,\ldots,r$\\[-1mm]
{\footnotesize 3:} & \quad Compute eigenvalues $\mu_1^{(1)},\ldots,\mu_{m_1}^{(1)}$ of $(A_1-\lambda_j B_1)-\mu C_1$ using Algo.~1.\\
{\footnotesize 4:} & \quad Compute eigenvalues $\mu_1^{(2)},\ldots,\mu_{m_2}^{(2)}$ of $(A_2-\lambda_j B_2)-\mu C_2$ using Algo.~1.\\
{\footnotesize 5:} & \quad Reorder eigenpairs:
 $|\mu_1^{(1)}-\mu_1^{(2)}|\le %|\mu_2^{(1)}-\mu_2^{(2)}|\le
 \cdots \le |\mu_m^{(1)}-\mu_m^{(2)}|$
for $m=\min(m_1,m_2)$.\\
{\footnotesize 6:} & \quad {\bf for} $k=1,\ldots,m$\\
{\footnotesize 7:} & \hbox{}\quad\quad {\bf if} $|\mu_k^{(1)}-\mu_k^{(2)}|< \delta$ {\bf then} add
  $(\lambda_j, \, \frac{1}{2}(\mu_k^{(1)}+\mu_k^{(2)}))$ to list of eigenvalues.\\
\end{tabular} \\
\vrule height 0pt depth 0.5pt width \textwidth
\medskip

Some remarks about Algorithm~2 are in order.
\begin{itemize}
\item If we know that each eigenvalue has a unique $\lambda$ component, then we can replace Lines 6 and 7 by selecting
  $(\lambda_j,{1\over 2}(\mu_1^{(1)}+\mu_1^{(2)}))$
  regardless of the difference $|\mu_1^{(1)}-\mu_1^{(2)}|$.
\item If $n_1=n_2=n$ then the complexity of Line 1 is ${\cal O}(n^6)$ while the complexity of
Lines 2 to 7 is at most ${\cal O}(n^5)$ in case $r={\cal O}(n^2)$.
\end{itemize}
%\medskip

\begin{example}\rm Consider a system of bivariate polynomials
(cf.~\cite[Exs.~5.3, 6.2, 6.4]{PHo16})
\begin{align*}
p_1(\lambda,\mu) & = 1 + 2\lambda + 3\lambda + 4\lambda^2 +
5\lambda \mu + 6\mu^2 + 7\lambda^3 + 8\lambda^2\mu +
9\lambda\mu^2 + 10\mu^3 = 0,\\
p_2(\lambda,\mu) & = 10 + 9\lambda + 8\mu + 7\lambda^2 + 6\lambda\mu +
 5\mu^2 + 4\lambda^3 + 3\lambda^2\mu + 2\lambda\mu^2 + \mu^3 = 0.
\end{align*}
Using a uniform determinantal representation from \cite{BDD17}, we
write the above system as a 2EP of the form
\begin{align*}
A_1+\lambda B_1+\mu C_1 &=\smath{\left[\begin{array}{ccccc}
   0 & 0 & 4 + 7\lambda & 1 & 0\\
   0 & 5 + 8\lambda & 2 & -\lambda & 1 \\
   6 + 9\lambda + 10\mu & 3 & 1 & 0 & -\lambda\\
   1 & -\mu & 0 & 0 & 0\\
   0 & 1 & -\mu & 0 & 0 \end{array}\right]}, \\
A_2+\lambda B_2+\mu C_2 &=\smath{\left[\begin{array}{ccccc}
   0 & 0 & 7 + 4\lambda & 1 & 0\\
   0 & 6 + 3\lambda & 9 & -\lambda & 1 \\
   5 + 2\lambda + \mu & 8 & 10 & 0 & -\lambda\\
   1 & -\mu & 0 & 0 & 0\\
   0 & 1 & -\mu & 0 & 0 \end{array}\right]},
\end{align*}
where $p_i(\lambda,\mu)=\det(A_i+\lambda B_i+\mu C_i)$ for $i=1,2$.
The obtained 2EP is singular and has 9 regular eigenvalues
$(\lambda_j,\mu_j)$ which are exactly the 9 solutions of the initial polynomial system.

If we apply Algorithm~2 to the above problem, we get all 9 solutions. In Line~2 we compute first components
$\lambda_1,\ldots,\lambda_9$ as finite eigenvalues of the corresponding singular pencil $\Delta_1-\lambda \Delta_0$ from
\eqref{Dgep}, whose KCF contains 4~$L_0$, 4~$L_0^T$, 2~$N_4$, 1~$N_2$, 2~$N_1$, and 9~$J_1$ blocks.
For each $\lambda_j$ we compute the candidates for $\mu_j$ in Lines 4 and 5, where
the KCF of singular pencils $(A_i-\lambda_jB_i)-\mu C_i$
contains 1~$N_2$ and $3$~$J_1$ blocks for $i=1,2$ and $j=1,\ldots,9$.

We remark that the above approach might also fail, in particular if we apply it to systems of bivariate polynomials of high degree.
Some of the eigenvalues of $\Delta_1-\lambda \Delta_0$ might be so ill-conditioned
that the algorithm cannot separate them from the infinite eigenvalues.
In such a case a possible solution would be to apply computation in higher precision, using, e.g.,
the Multiprecision Computing Toolbox \cite{Advanpix}.
\end{example}

\section{Conclusions}
\label{sec:concl}
We have proposed a method to approximate the finite eigenvalues of a singular
pencil by means of a \emph{rank-completing perturbation}.
%, i.e.,
%a random perturbation of the pencil of rank $k$,
%where $n-k$ is the normal rank of the pencil of dimension $n\times n$.
%While the idea of perturbing singular pencils has been studied before,
%this particularly useful perturbation is new.
The use of such a perturbation ensures that, generically, the finite and infinite eigenvalues remain fixed,
while there appear newly generated eigenvalues. For many problems we can well distinguish the
original eigenvalues from the newly created ones by considering
the angles of the eigenvectors with respect to the perturbation spaces,
and at the condition numbers of the eigenvalues.
Thus, this method may be useful for a wide range of
applications.

The proposed method could be an alternative
to % other available methods:
% such as
%When the eigenvalues of the original problem are very ill-conditioned,
%it may become difficult to distinguish between original and newly
%created eigenvalues, and it is a resource to have another alternative %type of methods available:
the class of staircase algorithms, such as e.g., Guptri~\cite{Guptri} or a staircase type algorithm
for singular two-parameter eigenvalue problems \cite{MPl10} in \cite{MultiParEig}.
These methods can be rapid and accurate, % for many problems.
however, the key part of staircase techniques are a number of rank decisions,
which can be difficult and ill-posed, see e.g., \cite{EM} and
Examples~\ref{ex:em} and \ref{ex:num3}.
In some cases, when these methods fail to return even a single eigenvalue,
the newly proposed method may still compute all or at least some of the eigenvalues.

A code for the approach developed in this paper is available in \cite{MultiParEig}.

\medskip\noindent{\bf Acknowledgments:}
The authors would like to thank Stefan Johansson for providing a beta
version of Matrix Canonical Structure (MCS) Toolbox \cite{MCS} which
includes a Matlab implementation of Guptri and is an important alternative
to the original Guptri \cite{Guptri} that we can no
longer use in Matlab due to the 32-bit limitation. Furthermore, the authors
would like to warmly thank two anonymous referees for their careful reading and
many expert suggestions and comments on a previous version of this paper.

\smallskip\noindent{\bf Genealogical acknowledgment:}
During this research project the first two authors found out that they are
twelfth cousins. Christian and Michiel thank their common ancestors
Caspar H\"olterhoff (1552--1625) and Catharina Teschemacher (ca.~1555--1639)
for making this possible.

\end{document}